\documentclass[12pt,oneside]{amsart}
\usepackage{amsmath}
\usepackage{amssymb}
\usepackage{amstext}
\usepackage{amscd}
\usepackage[all]{xy}
\usepackage{hyperref} 
\usepackage{soul}
\usepackage{color} 
\usepackage{typearea} 

\usepackage{pdfcomment}

\newtheorem{teo}{Theorem}[section]
\newtheorem{prop}[teo]{Proposition}
\newtheorem{lem}[teo]{Lemma}
\newtheorem{cor}[teo]{Corollary}
\newtheorem{conj}[teo]{Conjecture}

\newtheorem{defini}[teo]{Definition}

\newtheorem{rem}[teo]{Remark}

\makeatletter
\newcommand{\neutralize}[1]{\expandafter\let\csname c@#1\endcsname\count@}
\makeatother

\numberwithin{equation}{section}

\newcommand{\Hom}{\mbox{Hom}}

\newcommand{\IQbar}{\overline{\mathbb{Q}}}

\newcommand{\End}{{\rm End}}
\newcommand{\GSp}{{\rm GSp}}
\newcommand{\GL}{{\rm GL}}

\newcommand{\Lie}{\operatorname{Lie} }

\newcommand{\Gal}{{\rm Gal}}

\newcommand{\im}{{\rm im}}

\newcommand{\CC}{{\mathbb C}}
\newcommand{\RR}{{\mathbb R}}
\newcommand{\ZZ}{{\mathbb Z}}
\newcommand{\QQ}{{\mathbb Q}}

\newcommand{\GG}{{\mathbb G}}
\newcommand{\SSS}{{\mathbb S}}

\newcommand{\bU}{{\bf U}}

\newcommand{\cB}{{\mathcal B}}

\newcommand{\cA}{{\mathcal A}}

\newcommand{\cP}{{\mathcal P}}

\newcommand{\cO}{{\mathcal O}}

\newcommand{\oQ}{\overline{\QQ}}

\renewcommand{\cong}{\simeq}

\title[Hodge cycles and quadratic relation of CM periods]{Hodge cycles and quadratic relations between holomorphic periods on CM abelian varieties}
\author{Ziyang Gao, Emmanuel Ullmo}
\begin{document}

\address{Department of Mathematics, UCLA, Los Angeles, CA 90095, USA.}
\email{ziyang.gao@math.ucla.edu}

\address{IHES, Universit\'{e} Paris-Saclay, Laboratoire Alexandre Grothendieck; 35 route de Chartres, 91440 Bures sur Yvette, France.}
\email{ullmo@ihes.fr}

\maketitle

\begin{abstract} In this paper, we prove the following result advocating the importance of monomial quadratic relations between holomorphic CM periods. For any simple CM abelian variety $A$, we can construct a CM abelian variety $B$ such that all non-trivial Hodge relations between the holomorphic periods of the product $A\times B$ are generated by monomial quadratic ones which are also explicit. Moreover, $B$ splits over the Galois closure of the CM field associated with $A$.
\end{abstract}

\tableofcontents

\section{Introduction}
Let $\IQbar$ be the algebraic closure of $\QQ$ in $\CC$. All varieties in this paper are supposed to be defined over $\IQbar$ unless otherwise stated.

\subsection{Background}

Let $A$ be a square free abelian variety of dimension $g$ with complex multiplication by a CM algebra $E=\End(A)\otimes \QQ$. Then $[E:\QQ]=2g$.
Let $\cO_E$ be the ring of integers of $E$.  Let $\Phi=\{\phi_1,...,\phi_g\}$ be the associated CM-type. Then $\Hom(E,\CC)=\Phi\cup \overline{\Phi}$ and  $A$ is isogeneous to $\CC^g/\Gamma$ with $\Gamma=\Phi(\cO_E)$.

 Let $\{\omega_1,\dots,\omega_g\}$ be an $E$-eigenbasis of the space $\Omega_A$ of holomorphic one forms on $A$, \textit{i.e.} $e\cdot \omega_j = \phi_j(e)\omega_j$ for all $e \in E$ and $j \in \{1,\ldots,g\}$. Shimura \cite[Rmk. 3.4]{ShimuraPeriod} proved that  $\theta_j := \int_{\gamma}\omega_j$ is independent of the choice of $\gamma \in H_1(A,\ZZ)$  up to multiplication by an element in $\oQ$ and is non-zero for some $\gamma$. Thus we obtain $g$ non-zero complex numbers $\theta(A,\phi_1):=\theta_1, \ldots, \theta(A,\phi_g):= \theta_g$ well-defined up to $\IQbar^\times$. These numbers are called the \textit{holomorphic periods} of $A$.
 
 We also have the {\em antiholomorphic periods} $\theta(A,\overline{\phi}_j):= \int_{\gamma_j} \eta_j$ for  $\eta_j$ being the complex conjugation of $\omega_j$. To simplify notation, for two complex numbers $z_1$ and $z_2$, we write $z_1 \cong z_2$ if $z_2 = \alpha z_1$ for some $\alpha \in \IQbar^\times$. Then  $\theta(A,\phi_j) \theta(A,\overline{\phi}_j) \cong 2\pi i$ for each $j \in\{1,\ldots,g\}$, by the reciprocity law for the differential forms of the 1st and the 2nd kinds (see for example \cite[pp.36, equation~(3)]{BertrandReciprocityDiffForm}).

It is a fundamental question to study the transcendental properties of the periods, and we can often reduced to study the holomorphic periods $\theta_j = \theta(A,\phi_j) $ since $\theta(A,\phi_j) \theta(A,\overline{\phi}_j) \cong 2\pi i$. 
The expectation  is given by {\em Grothendieck's Period Conjecture}, which in our case predicts that
$\mathrm{trdeg}_{\IQbar}(2\pi i,\theta_1,\dots,\theta_g)=\dim \mathrm{MT}(A)$ and that all the algebraic relations between these periods are $\IQbar$-linear combinations of the {\em elementary relations}. Here $\mathrm{MT}(A)$ is the Mumford--Tate group of $A$. We refer to \cite[7.5]{AndreMotifs} for  a detailed discussion on the Grothendieck period conjecture. 

An algebraic relation between the $\theta_j$'s is called {\em elementary} if it is of the form 
\[
m_1(\theta_1,\dots,\theta_g) \cong m_2(\theta_1,\dots,\theta_g)
\]
for two monomials $m_1$ and $m_2$ in the $\theta_j$'s.

Grothendieck's Period Conjecture for CM abelian varieties is still widely open. 
The known results are:
\begin{enumerate}
\item (W\"ustholz \cite{WAST}) There are no $\oQ$-linear relations between the holomorphic periods, \textit{i.e.}
$$
\dim_{\oQ}\sum_{j=1}^g \oQ \theta_j=g.
$$
 In particular the $\theta_j$'s are transcendental numbers. 
\item (Deligne \cite{De1980}) The transcendence degree 
$$
\mathrm{trdeg}_{\IQbar}(2\pi i,\theta_1,\dots,\theta_g) \leq \dim \mathrm{MT}(A).
$$
\end{enumerate}

In \cite{GUY}, we proposed a framework to study the {\em quadratic relations} between the holomorphic periods. In particular, we proposed a {\em hyperbolic analytic subspace conjecture} \cite[Conj.~1.8]{GUY} -- the analogue of W\"ustholz's analytic subgroup theorem in the context of Shimura varieties, and explained \cite[Prop.~1.9]{GUY} how this conjecture implies that all quadratic relations between the holomorphic periods are $\IQbar$-linear relations of elementary quadratic relations, \textit{i.e.} those of the form $\theta_j\theta_{j'} \cong \theta_k\theta_{k'}$ with $\{j,j'\}\not=\{k,k'\}$. 

\subsection{Main result and digest}
The purpose of this text is to advocate for the importance of the  quadratic relations between the holomorphic periods. Our main result is the following theorem. It indicates that all non-trivial algebraic relations between holomorphic periods are generated by elementary quadratic relations, \textit{i.e.} are $\IQbar$-linear combinations of products of elementary quadratic relations; 
see $\S$\ref{SubsubsectionHodgeRelationDefnDigestIntro} for a precise meaning of this digest.

\begin{teo}\label{mainthm}
Let $A$ be a simple CM abelian variety associated with the CM field $E:= \mathrm{End}(A)\otimes \QQ$. Then 
 there exists a CM abelian variety $B$, split over $E^c$, such that all non-trivial Hodge relations between the holomorphic periods of $A\times B$ are generated by the elementary quadratic ones induced by $(2,2)$-Hodge cycles on $A\times B$.
\end{teo}
Moreover, given an $A$, the abelian variety $B$ and all the Hodge relations can be explicitly computed. They are constructed using CM abelian varieties of generalized anti-Weyl type, {\em with the quadratic elementary relations explicitly expressed in combinatoric terms}. See $\S$\ref{SubsectionProofMainThm}, Theorem~\ref{TheoremHodge22AntiWeyl} and \eqref{EqQuadraticHodgeRelation} for more details.

We need to explain many terminologies in Theorem~\ref{mainthm}. 
\begin{enumerate}
\item[(i)] The field $E^c$ is the Galois closure of the CM field $E$ in $\IQbar$;
\item[(ii)] Meaning of {\em $B$ is split over $E^c$}: Each CM subfield of the CM algebraic $\mathrm{End}(B)\otimes \QQ$ is isomorphic to a subfield of $E^c$;
\item[(iii)] The definition of {\em Hodge relations} will be given in Definition~\ref{DefnHodgeRelations}. Roughly speaking, it means all the algebraic relations induced by Hodge cycles on {\em any power} of $A\times B$. 
\end{enumerate}

To make our discussion more precise, let us introduce the ring homomorphism
\begin{equation}\label{EqEvaluationMapToDefineAlgRelation}
\mathrm{ev} \colon \IQbar[X_1,\ldots,X_g] \longrightarrow \CC, \qquad X_j \mapsto \theta_j.
\end{equation}
Its kernel is the ideal of all the algebraic relations between holomorphic periods, which we denote by $R_{\mathrm{alg}}$.

\subsubsection{DeRham--Betti classes  and elementary algebraic relations}
Each CM abelian variety $A$ is defined over $\IQbar$. The Betti cohomology $V:=H^1(A,\QQ)$ is a $\QQ$-vector space and the de Rham cohomology $W:=H^1_{\mathrm{dR}}(A)$ is a $\IQbar$-vector space. We have a comparison isomorphism
\begin{equation}\label{EqDeRhamBettiComparisonIntro}
\beta \colon  H^1_{\mathrm{dR}}(A)\otimes_{\IQbar} \CC=W_{\CC}\longrightarrow H^1(A,\QQ)\otimes_{\QQ}\CC=V_{\CC}.
\end{equation}

\begin{defini}
A {\em deRham-Betti class} of $A$ is a pair $(e,f)$ with $e\in W^{\otimes n}$, $f\in V_{\oQ}^{\otimes n}$ for some even integer $n$ such that 
\[
\beta^{\otimes n}(e)=(2\pi i)^{\frac{n}{2}}f.
\]
\end{defini}

We claim that each deRham-Betti class produces some elementary algebraic relation between the holomorphic periods of $A$, \textit{i.e.} a relation of the form 
\begin{equation}\label{eqdR-B}
m_1(\theta_1,\dots,\theta_g)= \lambda m_2(\theta_1,\dots\theta_g) 
\end{equation} 
with $\lambda\in \oQ$, and $m_1, m_2$ monomials in the $\theta_j$'s. Indeed, this follows immediately from the fact that $\beta$ is diagonalized to be $(\theta_1,\dots,\theta_g,\frac{2\pi i}{\theta_1},\dots,\frac{2\pi i}{\theta_g})$ under suitable $E$-eigenbases of $H_{\mathrm{dR}}^1(A)$ and  of $H^1(A,\QQ)$. We refer to $\S$\ref{SubsectionPeriodCM} for more details.

\subsubsection{Hodge cycles, Hodge relations, and various algebraic relations}\label{SubsubsectionHodgeRelationDefnDigestIntro}
Deligne \cite{De1980} proved that every Hodge cycle on an abelian variety is absolute Hodge. So for any pair of non-negative integers $(k,r)$,  any Hodge cycle in $H^{2k}(A^r,\QQ)\cap H_{\mathrm{dR}}(A^r,\CC)^{k,k}$  produces a deRham-Betti class, and thus induces an elementary algebraic relation between the holomorphic periods of $A$. 
Notice that some of the relations are trivial, for example any $(1,1)$-Hodge cycle on $A$ gives  relations of the form $\theta(A,\phi_j)\theta(A,\overline{\phi}_j) \cong 2\pi i$.  When passing to holomorphic periods we only get  the trivial equality $\theta_j = \theta_j$. 

Denote by $R_{\mathrm{Hodge}}$ the ideal generated by all such algebraic relations between the holomorphic periods of $A$. 

\begin{defini}\label{DefnHodgeRelations}
Each algebraic relation between the $\theta_j$'s induced by a Hodge cycle on $A^r$, for some $r \ge 1$, is called a {\em Hodge relation} between the $\theta_j$'s.
\end{defini}
While Hodge relations may arise from any power of $A$, in our terminology the quadratic ones arise from Hodge cycles on $A$. 
We also remark that different Hodge cycles on $A$ may give the same Hodge relation, for example two Hodge cycles $\alpha$ and $\alpha\wedge \alpha'$ with $\alpha'$ a $(1,1)$-Hodge cycle.

For each $r\ge 1$, any algebraic cycle on $A^r$ gives a Hodge cycle on $A^r$ and therefore a deRham-Betti class of $A$. Denote by $R_{\mathrm{alg-cycle}}$ (resp. $R_{\mathrm{dR-B}}$) the ideal of algebraic relations between holomorphic periods of $A$ given by all the  algebraic cycles on powers of $A$ (resp. by all deRham--Betti cycles on $A$). Denote also by $R_{\mathrm{elem}}$ the ideal of all elementary algebraic relations between holomorphic periods of $A$. Then we have
\begin{equation}
R_{\mathrm{alg-cycle}} \subseteq R_{\mathrm{Hodge}} \subseteq  R_{\mathrm{dR-B}} \subseteq R_{\mathrm{elem}}\subseteq R_{\mathrm{alg}}.
\end{equation}
The first inclusion is expected to be an equality as a consequence of the Hodge conjecture. All the other inclusions are expected to be equalities by Grothendieck's Period Conjecture. 

\subsubsection{}
Recall that a simple CM abelian variety $A$ of dimension $g$ is said to be {\em non-degenerate} if $\dim \mathrm{MT}(A)=g+1$ and to be {\em degenerate} otherwise. If $A$ is non-degenerate, then Hodge conjecture holds for $A$ and the Hodge ring is generated in degree $1$, \textit{i.e.} by classes of Cartier divisors. The relations given by Hodge cycles in $H^2(A,\QQ)\cap H^{1,1}(A,\CC)$ are the relations between holomorphic and anti-holomorphic period of the form 
$$
\theta(A,\phi)\theta(A,\overline{\phi})\simeq 2\pi i.
$$
 No Hodge relations between holomorphic periods exists, or equivalently $R_{\mathrm{Hodge}} = 0$ in this case.

If $A$ is degenerate, then $g+1> \dim \mathrm{MT}(A)\ge \log_2 g+2$ \cite[3.5]{RibetDivision-fields}, so there exist Hodge relations between holomorphic periods of degree  $\ge 2$ which are not obtained by cup product of classes of divisors. 

\subsection{Relation with the bi-$\IQbar$-structure on Shimura varieties developed in \cite{GUY}}
For a CM abelian variety $A$ of dimension $g\ge 1$, let $[o] \in\mathbb{A}_g(\IQbar)$ be a point parametrizing $A$. 

In \cite[Cor.~7.5]{GUY}, we proved that the existence of non-trivial elementary quadratic relations between holomorphic periods of $A$ is equivalent to the existence of bi-$\IQbar$-subspaces of $T_{[o]}\mathbb{A}_g$ which are not direct sum of root spaces. We go further and  propose the following conjecture.
\begin{conj}\label{ConjBiQbar}
The following are equivalent:
\begin{enumerate}
\item[(i)] There exist non-trivial elementary quadratic relations between the holomorphic periods of $A$;
\item[(ii)] There exist $r \ge 1$ and a Shimura subvariety $S$ of $\mathbb{A}_{gr}$ with the following property: $S$ contains the point $[\mathbf{o}] := ([o],\ldots,[o])$, and $T_{[\mathbf{o}] }S$ is not the direct sum of root spaces of $T_{[o]}\mathbb{A}_g$.
\end{enumerate}
\end{conj}
We refer to the comment above \cite[Cor.~7.5]{GUY} for the definition of root spaces.

We will prove this conjecture when $A$ is of generalized anti-Weyl type in $\S$\ref{SectionSubShimuraVariety}, where an explicit construction of a Shimura subvariety $S$ will be given based on some non-trivial quadratic Hodge relations between the holomorphic periods.

\subsection{Organization of the paper}
In $\S$\ref{SectionGeneralDiscussionCM}, we recall some basic facts about CM pairs, CM abelian varieties, their reflex pairs, and their periods. In particular, we will explain how to read off algebraic relations between their periods from the kernel of the reciprocity map. In $\S$\ref{SectionHodgeRing} we will recall the definition of Hodge rings of CM abelian varieties, and explain how to relate it to the kernel of the  reciprocity map. We also introduce one of the main tools used in our paper, a theorem of Pohlmann \cite[Thm.~1]{Poh}, to compute Hodge rings of CM abelian varieties.

The core of our paper is $\S$\ref{SectionWeylAntiWeyl}--\ref{SectionHodgeCyclesAntiWeylGen}. 

In $\S$\ref{SectionWeylAntiWeyl}, we start by reviewing the theory of CM abelian varieties of Weyl type \cite{CO2012}, sometimes called {\em Galois generic} \cite{CU2006}. They are simple CM abelian varieties with maximal Galois group and their Mumford--Tate groups are maximal tori. Then we will study their reflexes, called {\em of anti-Weyl type}, and compute the reflex CM types. An anti-Weyl CM abelian variety $A$ has dimension $2^{g-1}$ and its Mumford--Tate group $\mathrm{MT}(A)$ has dimension $g+1$, and thus the lower bound $\dim \mathrm{MT}(A) \ge \log_2 \dim A +2$ \cite[3.5]{RibetDivision-fields} is attained.  From the  Hodge theoretic point of view, CM abelian varieties of anti-Weyl type are thus very interesting. 
We will compute the kernel of the reciprocity map of CM abelian varieties of anti-Weyl type in $\S$\ref{SectionKernelReciprocity}. This kernel gives explicit algebraic relations between the holomorphic periods of $A$.  We find as a consequence, an explicit set $(\theta_0,\dots, \theta_g)$ of $g+1 = \dim \mathrm{MT}(A)$ periods of $A$ such that any other period of $A$ is algebraic over $\oQ(\theta_0,\dots, \theta_g)$.

In $\S$\ref{SectionAntiWeylGeneralisee}, we define  CM abelian varieties {\em of generalized anti-Weyl type}, using combinatoric data, and show that every simple CM abelian variety can be realized as a factor (up to isogeny) of a generlized anti-Weyl CM abelian variety. Then we compute Hodge ring and Hodge relations of generalized anti-Weyl type CM abelian varieties in $\S$\ref{SectionHodgeCyclesAntiWeylGen}. The computation is inspired by those in $\S$\ref{SectionKernelReciprocity}, and this ultimately proves our main result Theorem~\ref{mainthm}.

In $\S$\ref{SectionExample}, we will present an explicit computation for Theorem~\ref{mainthm}. More precisely we will present an example to show how to find the auxiliary abelian variety $B$ and to compute the Hodge relations in degree $2$ which generate all the Hodge relations between holomorphic periods of $A\times B$.

In the end in $\S$\ref{SectionSubShimuraVariety}, we relate the elementary quadratic relations between holomorphic CM periods to the theory of bi-$\IQbar$ decomposition of Shimura varieties which we developed in \cite{GUY}.

\subsection*{Acknowledgements}
We would like to thank Yves Andr\'{e} for many inspiring discussions on this subject, and Yunqing Tang on relevant discussion on Mumford--Tate groups of CM abelian varieties. We would also like to thank the referee for their careful reading and helpful comments. ZG would like to thank the IH\'{E}S for its hospitality during the preparation of this work.

\section{General discussion on CM abelian varieties and their periods}\label{SectionGeneralDiscussionCM}

In this section, we recall some standard facts and results on CM pairs and CM abelian varieties.

Let $A$ be a CM abelian variety of dimension $g$, \textit{i.e.} an abelian variety $A$ such that $\mathrm{End}(A)\otimes\QQ$ has a commutative $\QQ$-subalgebra of dimension $2g$.

\subsection{CM pairs}\label{SubsectionCMpair}
A number field $E$ is called a CM field if it is an imaginary quadratic extension of a totally real number field. A {\em CM algebra} is a finite product of CM fields. A {\em CM type} on a CM algebra $E$ is a subset $\Phi =\{\phi_1,\dots,\phi_g\} \subseteq \Hom(E,\CC)$ such that $\Hom(E,\CC)=\Phi \sqcup \overline{\Phi}$ (where $\overline{\Phi} = \{\bar{\phi}_1,\ldots,\bar{\phi}_g\}$ is the complex conjugate of $\Phi$).

A {\em CM pair} is a pair $(E,\Phi)$ consisting of a CM algebra $E$ and a CM type $\Phi$. Two CM pairs $(E,\Phi)$ and $(E',\Phi')$ are said to be isomorphic if there exists an isomorphism $\alpha \colon E \rightarrow E'$ of $\QQ$-algebras such that $\phi \circ \alpha \in \Phi$ whenever $\phi \in \Phi'$. To each CM pair $(E,\Phi)$, we can associate a CM abelian variety $A_{(E,\Phi)} := \CC^{[E:\QQ]/2}/\Phi(\cO_E)$ with $\cO_E$ the ring of integers of $E$. It is known that $A_{(E,\Phi)}$ is defined over $\IQbar$.

A CM pair $(E,\Phi)$ is said to be {\em primitive} if there does not exist a proper sub-CM algebra $E_0$ of $E$ such that $\Phi|_{E_0} := \{\phi|_{E_0} : \phi \in \Phi\}$ is a CM type on $E_0$. A CM pair $(E,\Phi)$ is primitive if and only if $A_{(E,\Phi)}$ is simple.

Finally, the {\em Galois closure} $E^c$ of a CM algebra $E$ is defined as follows:  $E = E_1^{n_1}\times \cdots \times E_m^{n_m}$ with each $E_j$ a CM field, and $E^c$ is defined to be the composite of the Galois closures of $E_1,\ldots,E_m$ in $\IQbar$. Then the Galois group $G := \mathrm{Gal}(E^c/\QQ)$  acts on $\Hom(E,\CC) = \Phi \sqcup \overline{\Phi}$.

\subsection{CM pair associated with $A$}\label{SubsectionCMType}
We say that $A$ is {\em associated with a CM pair $(E,\Phi)$}, if $A$ is isogeneous to $A_{(E,\Phi)} = \CC^{|\Phi|}/\Phi(\cO_E)$ with $\cO_E$ the ring of integers of $E$.

Assume $A$ is simple, then $E:= \mathrm{End}(A)\otimes \QQ$ is a CM field of degree $2g$. It is well-known that there exists a CM type $\Phi$ on $E$ such that $A$ is associated with the CM pair $(E,\Phi)$, this CM pair 
 $(E,\Phi)$ is primitive, and each primitive CM pair is obtained in this way.

In general,  $A$ is isogenous to $A_1^{n_1}\times \cdots \times A_m^{n_m}$ for some  pairwise non-isogenous simple abelian varieties $A_1, \ldots, A_m$. Each $A_j$ is associated with some CM pair $(E_j, \Phi_j)$ with $E_j$ a CM field. Set $E := E_1^{n_1} \times \cdots \times E_m^{n_m}$, with $E_j^{(k)}$ the $k$-th component of $E_j^{n_j}$ and $p_j^{(k)}$ the natural projection $E \rightarrow E_j^{(k)}$, and 
\begin{equation}\label{EqCMTypeNonSimple}
\Phi := \bigsqcup_{j=1}^m \Phi_j^{\sqcup n_j}, \quad \text{ where } \Phi_j^{\sqcup n_j} := \{\Phi_j \circ p_j^{(k)} : k \in \{1,\ldots,n_j\}\}.
\end{equation}
Then $(E,\Phi)$ is a CM pair and  $A$ is isogeneous to $A_{(E,\Phi)} = \CC^g/\Phi(\cO_E)$ with $\cO_E = \cO_{E_1}^{n_1}\times \cdots \times \cO_{E_m}^{n_m}$. So $A$ is associated with $(E,\Phi)$.

\subsection{Algebraic tori associated with $A$}\label{SubsectionAlgTori}
Use the notation from last subsection. 
In this subsection, we associate with $A$ three algebraic tori $\mathrm{MT}(A) \subseteq T \subseteq E^\times$ defined over $\QQ$. 

Let $V = H^1(A,\QQ)$. It is a $\QQ$-vector space of dimension $2g$ with a polarization $\psi$ which is a non-degenerate anti-symmetric pairing $V \times V \rightarrow \QQ$.

The first algebraic torus is $E^\times := (E_1^\times)^{n_1} \times \cdots \times (E_m^\times)^{n_m}$, where $E_j^\times := \mathrm{Res}_{E_j/\QQ}\GG_{\mathrm{m},E_j }$. Assume $m=1$ and $n_1 = 1$, \textit{i.e.} $A$ is simple. Then the action of $E= E_1$ on $V$ makes $V$ into a one dimensional $E$-vector space. Therefore  $E^\times$ acts on $V$  and  every character of $E^\times$ occurs with multiplicity $1$ in $V_{\CC}$. In this case, the group of characters of $E^\times$ is $X^*(E^\times)=\ZZ[\Hom(E,\CC)] = \oplus_{\phi \in \Phi} \ZZ \phi \bigoplus \oplus_{\phi\in \Phi} \ZZ \overline{\phi}$. For arbitrary $A$, we still have 
\[
X^*(E^\times) = \oplus_{\phi \in \Phi} \ZZ \phi \bigoplus \oplus_{\phi\in \Phi} \ZZ \overline{\phi}.
\]
Thus $E^\times$ can be identified with the diagonal torus of $\mathrm{GL}(V_{\CC})$, under
\begin{equation}\label{EqBettiCohDecomp1}
V_{\CC} = \oplus_{\phi \in \Phi} V_{\CC,\phi} \bigoplus \oplus_{\phi\in \Phi} V_{\CC,\overline{\phi}}.
\end{equation}

The second algebraic torus $T$ is defined as $E^\times \bigcap \GSp(V,\psi)$, and is a maximal torus of $\GSp(V,\psi)$ and hence has dimension $g+1$. The restriction of the morphism $\lambda \colon \GSp_{2g} \rightarrow \mathbb{G}_{\mathrm{m}}$, $h \mapsto \det(h)^{1/g}$, to $T$ is a character $\epsilon_0 \in X^*(T) = \Hom(T,\mathbb{G}_{\mathrm{m}})$.  By abuse of notation we write $\phi$ for $\phi|_T$ for each $\phi \in \Phi$.  The character group of $T$ is 
\begin{equation}
X^*(T) = \ZZ\epsilon_0\bigoplus \oplus_{\phi \in \Phi} \ZZ \phi.
\end{equation}
Note that $T$ splits over the Galois closure $E^c$ of $E$, \textit{i.e.} $T_{E^c} \cong \mathbb{G}_{\mathrm{m}, E^c}^{g+1}$. Hence the Galois group $\Gal(E^c/\QQ)$ acts on $X^*(T)$.

Before moving on, let us point out that the cocharacter group of $(E_1^\times)^{n_1} \times \cdots \times (E_m^\times)^{n_m}$, resp. of $T$, can be written in the same way as the character group, and hence contains the element $\epsilon_0$. 

\smallskip

Finally, let us turn to the Mumford--Tate group of $A$ denoted by $\mathrm{MT}(A)$, \textit{i.e.} the smallest $\QQ$-subgroup of $\GL(V)$ such that the map $\alpha \colon \SSS\longrightarrow \GL(V)_{\RR}$ defining the Hodge structure on $V$ factors through $\mathrm{MT}(A)_{\RR}$. 
It is known that $\alpha(\SSS) \subseteq ((E_1^\times)^{n_1} \times \cdots \times (E_m^\times)^{n_m})_{\RR}$ and $\alpha(\SSS)\subseteq \GSp(V,\psi)_{\RR}$. Hence $\mathrm{MT}(A) \subseteq T$. 

\begin{lem}\label{LemMTGroupCochar}
The cocharacter group $X_*(\mathrm{MT}(A))$ is the saturation in $X_*(T)$ of the $\ZZ[\Gal(E^c/\QQ)]$-submodule generated by  $ \sum_{\phi \in \Phi}\phi$.
\end{lem}
\begin{proof}
Let $\mu \colon \mathbb{G}_{\mathrm{m},\CC} \rightarrow \mathrm{GL}(V)_{\CC}$ be $z \mapsto \alpha_{\CC}(1,z)$, where $\alpha_{\CC}$ is the base change of the map $\alpha$ to $\CC$. By general knowledge of Hodge theory, $\mathrm{MT}(A)$ is the smallest $\QQ$-subgroup  of $\mathrm{GL}(V)$ such that $\mu$ factors through $\mathrm{MT}(A)_{\CC}$. A classical computation shows that $\mu=\sum_{\phi \in \Phi}\phi$ is a cocharacter of $T$. Hence the result follows because $T$ is split over $E^c$.
\end{proof}

\subsection{Reflex}\label{SubsectionReflex}
Let $(E,\Phi)$ be a CM pair, and let $E^c$ be the Galois closure of $E$ (see the end of $\mathsection$\ref{SubsectionCMpair} for definition). Then for each $\sigma \in G:=\mathrm{Gal}(E^c/\QQ)$, the set $\sigma\Phi := \{\sigma \circ \phi : \phi \in \Phi\}$ is still a CM type on $E$. Thus we can define the stabilizer  $\mathrm{Stab}_G(\Phi)$ of $\Phi$

The {\em reflex field} $E^*$ of $(E,\Phi)$ is defined to be $(E^c)^{\mathrm{Stab}_G(\Phi)}$. One can also characterise $E^*$ as the subfield of $E^c$  generated by the elements $\sum_{\phi \in \Phi} \phi(a)$ with all $a \in E$. Note that $\mathrm{Gal}(E^c/E^*) = \mathrm{Stab}_G(\Phi)$.

The {\em reflex CM type}  on $E^*$ is defined as follows. We have $\Hom(E^c, \CC) = \Hom(E^c, E^c) = G$.  Thus $\Phi_{E^c} := \{\sigma \in G : \sigma|_E \in \Phi\} \subseteq G$ is a CM type on $E^c$.  Hence under $\Hom(E^*, \CC) \cong G/\mathrm{Stab}_G(\Phi)$, the set $\Phi^* := \{\sigma^{-1} \mathrm{Stab}_G(\Phi) : \sigma\in G \text{ with } \sigma|_E \in \Phi\}$ is the reflex CM type on $E^*$. 

We call $(E^*,\Phi^*)$ the {\em reflex CM pair} of $(E,\Phi)$. Note that $E^*$ is  a CM field, in contrast to $E$. Better, $(E^*,\Phi^*)$ is  a primitive CM pair.

The {\em reflex norm} is the morphism
\begin{equation}\label{EqReflexNorm}
\mathrm{rec} \colon (E^*)^{\times}  \rightarrow E^\times, \quad a \mapsto \prod\nolimits_{\psi \in \Phi^*} \psi(a).
\end{equation}
Notice that this morphism can be viewed either as a homomorphism of groups or a morphism of algebraic tori defined over $\QQ$.

Now assume  $(E,\Phi)$ is a primitive CM pair. Then $(E^{**}, \Phi^{**}) = (E,\Phi)$. Moreover, $\mathrm{MT}(A_{(E,\Phi)}) \cong \mathrm{MT}(A_{(E^*,\Phi^*)})$, and the image of $\mathrm{rec}$ is $\mathrm{MT}(A_{(E,\Phi)})$.

\subsection{Periods of CM abelian varieties}\label{SubsectionPeriodCM}

Let $A$ be a CM abelian variety of dimension $g$, associated with the CM pair $(E,\Phi)$ as constructed in $\mathsection$\ref{SubsectionCMType}. The deRham cohomology  $H_{\mathrm{dR}}^1(A)$  is a $\IQbar$-vector space since $A$ is defined over $\IQbar$, and it has an  $E$-eigenbasis $\{\omega_1,\ldots,\omega_g,\eta_1,\ldots,\eta_g\}$. As in \eqref{EqBettiCohDecomp1}, the Betti cohomology $H^1(A,\IQbar) = H^1(A,\QQ)\otimes\IQbar$ has an $E$-eigenbasis $\{\phi_1,\ldots,\phi_g,\overline{\phi}_1,\ldots,\overline{\phi}_g\}$.

The period torsor $\mathcal{B}_A$ is a $\mathrm{MT}(A)$-torsor representing 
$
\underline{\mathrm{Iso}}^{\otimes}(H_{\mathrm{dR}}^1(A), H^1(A,\IQbar))
$. 
 The de Rham--Betti comparison isomorphism
$$
\beta \colon H_{\mathrm{dR}}^1(A_{\CC})\longrightarrow H^1(A,\CC)
$$
is a complex point of $\mathcal{B}_A$.  We computed in \cite[Rmk.~6.5]{GUY} that, up to scalars in $\IQbar^\times$, $\beta$ is diagonalized under the basis $\{\omega_1,\ldots,\omega_g,\eta_1,\ldots,\eta_g\}$ of $H_{\mathrm{dR}}^1(A)$ and the basis 
$\{\phi_1,\ldots,\phi_g, \overline{\phi}_1,\ldots,\overline{\phi}_g\}$ of $H^1(A,\IQbar)$ with coordinates
\begin{equation}\label{EqBetaDiag}
\left(\theta_1,\dots, \theta_g, \frac{2\pi i}{\theta_1}, \dots,  \frac{2\pi i}{\theta_g} \right).
\end{equation}
The map $(\sum_j a_j\omega_j + \sum_l b_l\eta_l \mapsto \sum_j a_j \phi_j + \sum_l b_l \overline{\phi}_l)$ is a $\IQbar$-point of $\mathcal{B}_A$. Using this as a base point, we get an isomorphism  $ \mathcal{B}_A \xrightarrow{\sim} \mathrm{MT}(A)_{\IQbar}$ which sends $\beta$ to the diagonal matrix \eqref{EqBetaDiag}. By abuse of notation, denote this diagonal matrix still by $\beta$. Then $\beta \in \mathrm{MT}(A)(\CC)$, and we have
\begin{equation}\label{EqBetaEvaChar}
\phi_j(\beta) =\theta_j ~~(\forall j), \quad \overline{\phi}_j(\beta) = 2\pi i/\theta_j ~~(\forall j), \quad \epsilon_0(\beta) = 2\pi i
\end{equation}
with all $\phi_j$, $\overline{\phi}_j$ viewed as characters of $E^\times$ (or of $T$) and $\epsilon_0 \in X^*(T)$.

The following lemma is an immediate consequence of the discussion above. Consider  the  exact sequences obtained from $\mathrm{MT}(A) < T$ and $\mathrm{MT}(A) < E^\times$ respectively
 \begin{align}\label{EqSES}
 0\longrightarrow N\longrightarrow X^*(T)\longrightarrow X^*(\mathrm{MT}(A))\longrightarrow 0 \\ 
  0\longrightarrow N'\longrightarrow X^*(E^\times)\longrightarrow X^*(\mathrm{MT}(A))\longrightarrow 0. \nonumber
 \end{align}
Any non-zero vector in $N$ (resp. in $N'$) produces a non-trivial monomial relation between the holomorphic periods of $A$ and $\pi$ (resp. between the holomorphic and the anti-holomorphic periods of $A$ and $\pi$). More precisely, let $M_{\mathrm{hol}}$ be the multiplicative group of monomials in the ideal of algebraic relations between $\theta_1,\ldots,\theta_g,\pi$, and let $M_{\mathrm{full}}$ be the multiplicative group of monomials in the ideal of algebraic relations between $\theta_1,\ldots,\theta_g,\frac{2\pi i}{\theta_1},\ldots,\frac{2\pi i}{\theta_g}, \pi$.

\begin{lem}
There are (natural) injective group homomorphisms $N \rightarrow M_{\mathrm{hol}}$ and $N' \rightarrow M_{\mathrm{full}}$.
\end{lem}

In particular, if $T=\mathrm{MT}(A)$ no relations among the holomorphic periods are obtained from $N$, and the set of relations among the holomorphic and anti-holomorphic periods are generated by  $\theta_j \frac{2\pi i}{\theta_j}\cong 2\pi i$ for $j \in \{1,\ldots,g\}$.

In the case where $A$ is simple, we can do better.

\begin{lem}\label{LemNprimeReflex}
Assume that $A$ is a simple CM abelian variety. Then the map $X^*(E^\times) \rightarrow X^*(\mathrm{MT}(A))$ factors through $\mathrm{rec}^*$ obtained from the reflex norm \eqref{EqReflexNorm}, and $N' = \ker(\mathrm{rec}^*)$. In other words, we have the short exact sequence
\[
0 \rightarrow N' \rightarrow X^*(E^\times) \xrightarrow{\mathrm{rec}^*} X^*((E^*)^\times).
\]
\end{lem}
\begin{proof}
If $A$ is simple, then $(E,\Phi)$ is primitive. So the conclusion follows from the last paragraph of $\mathsection$\ref{SubsectionReflex}.
\end{proof}

\section{Hodge ring}\label{SectionHodgeRing}

Let $A$ be a CM abelian variety, associated with the CM pair $(E,\Phi)$. 
Let $E^c$ be the Galois closure of $E$ as defined at the end of $\mathsection$\ref{SubsectionCMpair}, and let $G = \mathrm{Gal}(E^c/\QQ)$ be the Galois group.

\subsection{Hodge rings of CM abelian varieties}\label{s2.1}
Let $B^p(A):=H^{p,p}(A,\CC)\cap H^{2p}(A,\QQ)$ be the $\QQ$-vector space of  $\QQ$-Hodge cycles of type $(p,p)$ on $A$. The \textit{Hodge ring} of $A$ is defined to be
\begin{equation}
B(A) := \bigoplus_{p=0}^{g}B^p(A).
\end{equation}

Let us recall the following description of $B(A)$ by Pohlmann's Theorem \cite{Poh}.  
Let $S=\Phi \sqcup \overline{\Phi}$. 
We have an isomorphism $H^r(A,\CC)=\bigwedge^r \CC^S$ and 
$$
H^{1,0}(A,\CC)= \sum_{\phi \in \Phi} \CC \phi \mbox{ and } H^{0,1}(A,\CC)= \sum_{\phi' \in \overline{\Phi}} \CC \phi'. 
 $$

Fix an ordering on $S$ such that $\phi \le \phi'$ for each $\phi \in \Phi$ and $\phi' \in \overline{\Phi}$. 
Let $\cP(S) $ be the set of subsets of $S$.
For each subset $P\in \cP(S)$ , let  $[P]:=\bigwedge_{\phi\in P}\phi.$  Then the  component 
$$
H^{p,q}(A,\CC)= \bigwedge^p H^{1,0}(A,\CC)\otimes \bigwedge^q H^{0,1}(A,\CC)
$$
of the Hodge decomposition has a basis consisting of the $[P]$ such that $\vert P\cap \Phi\vert=p$ and $\vert P\cap \overline{\Phi}\vert=q$ for $P$ an ordered subset of $S$.  Finally, the Galois group $G=\Gal(E^c/\QQ)$ operates on $\cP(S)$.
\begin{teo}[Pohlmann]\label{Poh}
For each $p \ge 0$, the vector space $B^p(A)\otimes \CC$ has a basis consisting of $[P]$ for those ordered sets $P\in\cP(S)$ with $|P| = 2p$ such that 
\begin{equation}\label{toto}
\vert \sigma P\cap \Phi\vert=\vert \sigma P\cap \overline{\Phi}\vert \quad \text{ for all } \sigma \in G.
\end{equation}
In particular $\dim_{\QQ}B^p(A)$ is the number of ordered $P\in \cP(S)$ with $\vert P\vert=2p$ satisfying  \eqref{toto}.
\end{teo} 
\begin{proof}
Pohlmann \cite[Thm.~1]{Poh} states this result when $A$ is simple. The proof remains valid for an arbitrary CM abelian variety $A$. To make the paper more self-contained, we include the proof here.

Each element of $B^p(A) \subseteq B^p(A)\otimes \CC$ can be written as a linear combination $f = \sum_j c_j[P_j] \in B^p(A)$ with $c_j \in \CC$ and $P_j \in \cP(S)$ with $|P_j| = 2p$. For each $\tau \in \mathrm{Aut}(\CC/\QQ)$, we then have $\sum_j \tau(c_j)[\tau P_j] = \tau(f) = f \in H^{p,p}(A,\CC)$. So each $P_j$ satisfies \eqref{toto}. Hence every element of $B^p(A)$ is a $\CC$-linear combination of $[P]$ with $P \in \cP(S)$ satisfying \eqref{toto}, and so the same holds true for every element of $B^p(A)\otimes \CC$.

It remains to prove that $[P] \in B^p(A)\otimes \CC$ for each $P \in \cP(S)$ with $|P|=2p$ satisfying \eqref{toto}. Let $\{u_1,\ldots,u_r\}$ be a basis of $E^c$ over $\QQ$, with $r=[E^c:\QQ]$. Set $f_j := \sum_{\sigma \in G}\sigma(u_j)[\sigma P]$ for each $j\in \{1,\ldots,r\}$. Then $\sigma(f_j) = f_j$ for each $\sigma \in \mathrm{Gal}(E^c/\QQ)$, and $f_j \in H^{p,p}(A,\CC)$ by \eqref{toto}. So $f_j \in B^p(A)$. But $\det\left( \sigma(u_j) \right)_{\sigma,j} \not= 0$, we can solve the linear system $f_j = \sum_{\sigma \in G}\sigma(u_j)[\sigma P]$, and find that $[\sigma P] \in B^p(A) \otimes \CC$ for each $\sigma \in G$. In particular, $[P] \in B^p(A)\otimes \CC$. We are done.
\end{proof}

\subsection{Kernel of the restriction of character groups}

Consider the short exact sequence
\begin{equation}\label{EqKernelMaxTorusMT}
0 \rightarrow N \rightarrow X^*(T) \rightarrow X^*(\mathrm{MT}(A)) \rightarrow 0.
\end{equation}

Each element $\alpha \in X^*(T)$ can be written in a unique way as $\sum_{\phi \in \Phi} a_{\phi} \phi+a\epsilon_0$ with $a_{\phi} \in \ZZ$, and we denote by $n(\alpha) := \max\{|a_{\phi}|\}$.

\begin{lem}\label{LemmaInnerProductOnKernel}
Any element in $N$ can be  written in the form $\sum_{\phi \in \Phi} a_{\phi} \phi$ with $\sum_{\phi \in \Phi} a_{\phi}=0$. 
\end{lem}
\begin{proof}
 Define an inner product $\langle \cdot, \cdot \rangle$ on  $\bigoplus_{\phi \in \Phi \cup \overline{\Phi} } \ZZ \phi$ by setting $\langle \phi, \phi' \rangle$ to be $1$ if $\phi = \phi'$ and to be $0$ otherwise. Recall that $X_*(\mathrm{MT}(A))$ is generated by $\mathrm{Gal}(E^c/\QQ)\sum_{\phi\in \Phi}\phi$ by Lemma~\ref{LemMTGroupCochar}. Then $N$ is orthogonal to both $\sum_{\phi\in \Phi}\phi$ and to $\sum_{\phi \in \Phi} \overline{\phi} = \rho\sum_{\phi\in \Phi}\phi$ for the complex conjugation $\rho$. The second orthogonality condition implies that the coefficient of $\epsilon_0$ is $0$, then the first implies $\sum a_{\phi}=0$.
\end{proof}
\medskip
 For each $n \ge 1$, define 
\[
X^*(T)_n := \{\alpha \in X^*(T) : n(\alpha) \le n \}.
\]
Using the notation from $\S$\ref{SubsectionCMType}, the CM abelian variety $A^n$ has CM type $\Phi^{\sqcup n}$. For each $\ell \in \{1,\ldots,n\}$, denote by $\Phi^{(\ell)}$ the $\ell$-th component $\Phi$ in $\Phi^{\sqcup n}$, and for each $\phi \in \Phi$ denote by $\phi^{(\ell)}$ the element $\phi$ in  $\Phi^{(\ell)}$. Fix an ordering on $\Phi \cup \overline{\Phi}$ such that $\phi \le \phi'$ for each $\phi \in \Phi$ and $\phi' \in \overline{\Phi}$. It induces an ordering on $\Phi^{\sqcup n} \cup \overline{\Phi}^{\sqcup n}$ as follows:
\begin{itemize}
\item On each $\Phi^{(\ell)} \cong \Phi$ (resp. $\overline{\Phi}^{(\ell)}\cong \overline{\Phi}$) it is the one on $\Phi$ (resp. on $\overline{\Phi}$);
\item For $\ell \le \ell'$, we have $\phi \le \phi'$ for all $\phi \in \Phi^{(\ell)}$ and $\phi \in \Phi^{(\ell')}$;
\item For $\ell \le \ell'$, we have $\phi \le \phi'$ for all $\phi \in \overline{\Phi}^{(\ell)}$ and $\phi \in \overline{\Phi}^{(\ell')}$;
\item We have $\phi \le \phi'$ for all $\phi \in \Phi^{\sqcup n}$ and $\phi' \in  \overline{\Phi}^{\sqcup n}$.
\end{itemize}

To each $\alpha = \sum_{\phi \in \Phi} a_{\phi} \phi \in  X^*(T)_n$, we can assign several subsets $ P_{\alpha}$ of $\Phi^{\sqcup n} \cup \overline{\Phi}^{\sqcup n}$
\begin{equation}\label{EqMapKernelHodgeRing}
P_{\alpha}= \bigcup_{a_{\phi} > 0, ~ j = 1, \ldots, a_{\phi}}  \{\phi^{(\ell_j)}\} \bigcup_{a_{\phi} < 0, ~ j = 1, \ldots, |a_{\phi}|}  \{\overline{\phi^{(\ell_j)}}\}
\end{equation}
with all the $\ell_j$'s taking values in $\{1,\ldots,n\}$. If $n=1$, then there is a unique way to assign a subset of $\Phi \cup \overline{\Phi}$. When $n \ge 2$, the assignment is not unique, and one possible choice is $P_{\alpha}:=P_{\alpha}^{(1)} \cup P_{\alpha}^{(2)} \cup \dots \cup P_{\alpha}^{(n)}$ with, for each $\ell \in \{1,\ldots,n\}$, 
$$
P_{\alpha}^{(\ell)} =\{\phi \in \Phi^{(\ell)}: a_{\phi} \ge \ell\} \cup \{\overline{\phi} \in \overline{\Phi^{(\ell)}}: a_{\phi} \le -\ell \} \in \cP(\Phi^{(\ell)} \cup \overline{\Phi^{(\ell)}}).
$$
\begin{prop}\label{PropKernelHodge}
Let $\alpha \in N \cap X^*(T)_n$, and define $P_{\alpha} \subseteq \Phi^{\sqcup n} \cup \overline{\Phi}^{\sqcup n}$ as above. Then the element
\[
[P_{\alpha}] :=  \bigwedge\nolimits_{\phi \in P_{\alpha}}\phi
\]
is in $B(A^n)\otimes \CC$. 
Moreover, $B(A^n) \otimes \CC$ (viewed as a $\CC$-vector space) has a basis consisting of $[P_{\alpha}]$ for all $\alpha \in N\cap X^*(T)_n$ and of $\phi^{(\ell)} \wedge \overline{\phi^{(\ell')}}$ for all $\phi \in \Phi$ and $\ell,\ell' \in \{1,\ldots,n\}$.
\end{prop}
\begin{proof}
For each $\alpha =\sum_{\phi \in \Phi} a_{\phi} \phi  \in X^*(T)_n$ and each $P_{\alpha}$ of the form \eqref{EqMapKernelHodgeRing}, we start by proving that
\begin{equation}\label{EqKernelHodgeRing}
\alpha\in N \Longleftrightarrow   
 \vert \sigma P_{\alpha} \cap \Phi^{\sqcup n} \vert=\vert \sigma P_{\alpha} \cap \overline{\Phi}^{\sqcup n} \vert \text{ for all }\sigma \in G\text{ and } \sum\nolimits_{\phi \in \Phi} a_{\phi} = 0.
\end{equation}
Consider the inner product $\langle \cdot, \cdot \rangle$ on  $\bigoplus_{\phi \in \Phi \cup \overline{\Phi} } \ZZ \phi$ defined at the beginning of the proof of Lemma~\ref{LemmaInnerProductOnKernel}. Then for any $\sigma \in G$, we have 
$$
\langle \sum_{\phi \in \Phi} a_{\phi} \phi , \sigma\sum_{\phi\in \Phi} (\phi - \overline\phi) \rangle = |P_{\alpha} \cap \sigma \Phi^{\sqcup n} | - |P_{\alpha} \cap \sigma \overline{\Phi}^{\sqcup n}|. 
$$

 Recall that $X_*(\mathrm{MT}(A))$ is generated by $\mathrm{Gal}(E^c/\QQ)\sum_{\phi\in \Phi}\phi$ by Lemma~\ref{LemMTGroupCochar}. 
 Hence
\begin{align*}
\sum_{\phi \in \Phi} a_{\phi} \phi  \in N & \Longleftrightarrow \langle \sum_{\phi \in \Phi} a_{\phi} \phi , \sigma\sum_{\phi\in \Phi}\phi\rangle = 0 \text{ for all }\sigma \in G \\
& \Longleftrightarrow  \langle \sum_{\phi \in \Phi} a_{\phi} \phi , \sigma\sum_{\phi\in \Phi} (\phi - \overline\phi) \rangle = 0 \text{ and } \langle \sum_{\phi \in \Phi} a_{\phi} \phi , \sigma\sum_{\phi\in \Phi} (\phi + \overline\phi) \rangle = 0 \text{ for all }\sigma \in G\\
& \Longleftrightarrow  \langle \sum_{\phi \in \Phi} a_{\phi} \phi , \sigma\sum_{\phi\in \Phi} (\phi - \overline\phi) \rangle = 0 \text{ for all }\sigma \in G\text{ and } \langle \sum_{\phi \in \Phi} a_{\phi} \phi , \sum_{\phi \in \Phi\sqcup \overline{\Phi}}\phi \rangle = 0 \\
& \Longleftrightarrow  \langle \sum_{\phi \in \Phi} a_{\phi} \phi , \sigma\sum_{\phi\in \Phi} (\phi - \overline\phi) \rangle = 0 \text{ for all }\sigma \in G\text{ and } \sum_{\phi \in \Phi} a_{\phi} = 0 \\
& \Longleftrightarrow |P_{\alpha} \cap \sigma \Phi^{\sqcup n} | - |P_{\alpha} \cap \sigma \overline{\Phi}^{\sqcup n}| = 0 \text{ for all }\sigma \in G\text{ and } \sum\nolimits_{\phi \in \Phi} a_{\phi} = 0.
\end{align*}
Thus we have established \eqref{EqKernelHodgeRing}. Now the last condition implies $[P_{\alpha}] \in B(A^n)\otimes \CC$ by Pohlmann Theorem~\ref{Poh}. So we have established the first assertion of the proposition.

Next by Theorem~\ref{Poh}, the $\CC$-vector space $B(A^n) \otimes \CC$ has a basis consisting of those $[P]$, with $P  \subseteq S^{\sqcup n} = \Phi^{\sqcup n} \cup \overline{\Phi}^{\sqcup n}$, such that $\vert \sigma P \cap \Phi^{\sqcup n} \vert=\vert \sigma P  \cap \overline{\Phi}^{\sqcup n} \vert$ for all $\sigma \in G$. For each such $P$, define
\[
\alpha(P) := \sum_{\phi \in \Phi} \#\{\ell : \phi \in P\cap \Phi^{(\ell)}\} \phi - \sum_{\phi \in \Phi}  \#\{\ell : \phi \in P\cap \overline{\Phi^{(\ell)}}\}  \phi \in X^*(T)_n
\]  
Then by construction, there exists a subset $P_{\alpha(P)}$ of the form \eqref{EqMapKernelHodgeRing}, assigned to $\alpha(P)$, such that $P \setminus P_{\alpha(P)} = \{\phi_{i_1}^{(\ell_1)}, \overline{\phi_{i_1}^{(\ell'_1)}}, \ldots, \phi_{i_r}^{(\ell_r)},\overline{\phi_{i_r}^{(\ell'_r)}}\}$ for some $\phi_{i_1},\ldots,\phi_{i_r} \in \Phi$. Thus 
\[
\vert \sigma P_{\alpha(P)} \cap \Phi^{\sqcup n} \vert = \vert \sigma P \cap \Phi^{\sqcup n} \vert - r = \vert \sigma P \cap \overline{\Phi}^{\sqcup n} \vert - r = \vert \sigma P_{\alpha(P)} \cap \overline{\Phi}^{\sqcup n} \vert 
\]
 for all $\sigma \in G$. On the other hand, 
 \[
 \sum_{\phi\in \Phi}\#\{\ell : \phi \in P\cap \Phi^{(\ell)}\} = \vert  P \cap \Phi^{\sqcup n} \vert=\vert  P  \cap \overline{\Phi}^{\sqcup n} \vert = \sum_{\phi\in\Phi}\#\{\ell : \phi \in P\cap \overline{\Phi^{(\ell)}}\} .
 \] 
So  $\alpha(P) \in N$ by \eqref{EqKernelHodgeRing}. Now we can conclude.
\end{proof}

\section{Weyl and anti-Weyl type CM abelian varieties}\label{SectionWeylAntiWeyl}

\subsection{CM Abelian Varieties of Weyl types}\label{s1.3}
Let $E$ be a CM field of degree $2g$, and let $E^c$ be the Galois closure of $E$ and $G$ be the Galois group $G=\Gal(E^c/\QQ)$. Let $\Phi$ be a CM type on $E$. 
Then $G$ can be seen as a subgroup of $(\ZZ/2\ZZ)^g\rtimes S_g$ by \cite[Imprimitivity Theorem]{DodsonThe-structure-o}.

\begin{defini}
We say that $E$ is {\em of Weyl type}  if $\Gal(E^c/\QQ)=(\ZZ/2\ZZ)^g\rtimes S_g$. 
We also say that a CM abelian variety $A$ is {\em of Weyl type} if $\mathrm{End}(A)\otimes \QQ$ is a CM field of Weyl type. 
\end{defini}
Assume $E$ is of Weyl  type. Then any CM type $\Phi'$ of $E$ is conjugate of $\Phi$ by an element of $G$. Thus $G$ acts transitively on the set of all CM types on $E$. This justifies that a CM field $E$ being of Weyl type is independent of the choice of $\Phi$. Moreover, a CM abelian variety of Weyl type is by definition simple, and thus the associated CM pair $(E,\Phi)$  is a primitive CM pair.

 In the rest of this subsection, we will identify $S_g$ with $\{0\}\times S_g < G$ to ease notation.

Let $F$ be the totally real field of degree $g$ contained in $E$ and let $F^c$ be the Galois closure of $F$. 
Then 
\begin{equation}
\Gal(F^c/\QQ)=S_g \mbox{ and } \Gal(E^c/E)=\{0\}\times (\ZZ/2\ZZ)^{g-1}\rtimes S_{g-1}.
\end{equation}

The Galois group $G$ acts transitively on 
$\Hom (E,\CC)=\Phi \cup \overline{\Phi}$ where $\Phi=\{\phi_1,\dots, \phi_g\}$
and $\overline{\Phi}= \{\overline{\phi}_1, \dots, \overline{\phi}_g \}$. By convention $\phi_1$ is the identity map giving the inclusion $E\subseteq E^c\subseteq \CC$.

The action of  $S_g$ 
preserves  $\Phi$ and $\overline{\Phi}$, and $S_g$
acts by permutation of the indexes of elements of $\Phi$ and of $\overline{\Phi}$. Let $\{e_1,\ldots,e_g\}$ be the canonical basis of $(\ZZ/2\ZZ)^g$. Then the action of $(\ZZ/2\ZZ)^g$ on $\Phi$ is determined by
\begin{equation}\label{EqActionOfZ2gOnWeyl}
e_j \phi_k = \begin{cases} \overline{\phi}_j & \text{ if } k=j \\ \phi_k & \text{ if }k\not=j.\end{cases}
\end{equation}
In particular $\rho:=e_{\{1,\dots,g\}}$ acts as the complex conjugation. 

Let 
\begin{equation}\label{EqStabPhi1}
H:= \mbox{Fix}(\phi_1)=\Gal (E^c/E)= \{0\}\times (\ZZ/2\ZZ)^{g-1}\rtimes S_{g-1}.
\end{equation}
Let $\sigma=(1,2,3,\dots,g)$ be the standard $g$-cycle in $S_g$.  Then 
\[
 G/H = \{H, \sigma H, \sigma^2 H, \dots,  \sigma^{g-1} H, \rho H, \sigma\rho H, \dots,  \sigma^{g-1} \rho H\}
\]
and $G$ acts on $G/H$ by multiplication on the left. 
We have a natural bijection
\begin{equation}
\eta \colon \Hom(E,\CC) \xrightarrow{\sim} G/H
\end{equation}
which sends $\phi_j\mapsto \sigma^{j-1}H$ and $\overline{\phi}_j \mapsto \sigma^{j-1}\rho H$ (for all $j \in \{1,\ldots,g\}$).

\begin{lem}\label{LemmaLeftRightAction}
The bijective map $\eta$ is $G$-equivariant.
\end{lem}
\begin{proof}
It suffices to check that $\eta$ is equivariant for all $\sum_{k\in I}e_k$ with  $I \subseteq \{1,\ldots,g\}$ and for all $\alpha \in S_g$. 

We start with $\sum_{k\in I}e_k$. We have seen that $(\sum_{k\in I}e_k) \cdot \phi_j$ equals $\overline{\phi}_j$ if $j \in I$ and equals $\phi_j$ otherwise.
On the other hand  
\[
(\sum_{k\in I}e_k) \sigma^{j-1} H=  \sigma^{j-1} (\sum_{k\in \sigma^{1+g-j}(I)}e_k) H= \begin{cases} 
\sigma^{j-1} H & \mbox{ if } 1\notin \sigma^{1+g-j}(I) \\
 \sigma^{j-1} \rho H & \mbox{ if } 1\in \sigma^{1+g-j}(I).
\end{cases}
\]
The condition $1\not\in \sigma^{1+g-j}(I)$ is equivalent to $\sigma^{j-1}(1)=j\notin I$. Hence we are done for elements in $(\ZZ/2\ZZ)^g$.

Next take $\alpha \in S_g$. Then $\alpha \cdot \phi_j = \phi_{\alpha(j)}$. On the other hand, denote by $\beta = \alpha\sigma^{j-1}$. Then $\beta \in  \sigma^{j_\beta-1} H$ for some $j_{\beta}$. Thus $\alpha = \sigma^{j_{\beta}-1} h  \sigma^{1+g-j}$ for some $h \in H$, and hence $
\alpha(j) = \sigma^{j_{\beta}-1} h(1) = \sigma^{j_{\beta}-1}(1) = j_{\beta}$. So
\[
\alpha \sigma^{j-1}H   = \beta H = \sigma^{j(\alpha)-1} H.
\]
We are done.
\end{proof}

\subsection{CM abelian varieties of anti-Weyl type}\label{s1.4}
\begin{defini}
A CM field  is said to be {\em of anti-Weyl type} if it is the reflex field of a CM pair $(E,\Phi)$ with $E$ of Weyl type. A CM abelian variety  is said to be {\em of anti-Weyl type} if it is associated with the reflex of a CM pair $(E,\Phi)$ with $E$ of Weyl type.
\end{defini}
Let $(E,\Phi)$ be a CM pair with $E$ of Weyl type, and let $(E^*,\Phi^*)$ be its reflex CM pair. The stabilizer of $\Phi$ is the subgroup $H'=\{0\}\times S_g$ of $G = \mathrm{Gal}(E^c/\QQ)$, therefore  $E^* = (E^c)^{H'}$. 
Now we compute the reflex CM type $\Phi^*$ on  $E^*$.

 For each $I \subseteq \{1,\ldots,g\}$, denote by $\varepsilon_I := \sum_{j \in I} e_j$. Then the action of $G$ on $\Hom(E^*,\CC)$ gives a $G$-equivariant bijection 
\begin{equation}
\eta^* \colon \Hom (E^*,\CC) \xrightarrow{\sim} G/ H' =\{f H'\}_{f\in (\ZZ/2\ZZ)^g}= \{ \varepsilon_I H' \}_{I\subseteq\{1,\dots,g\}}.
\end{equation}
By abuse of notation, let us identify $\Hom(E^*,\CC)$ with $G/ H' =\{f H'\}_{f\in (\ZZ/2\ZZ)^g}= \{ \varepsilon_I H' \}_{I\subseteq\{1,\dots,g\}}$ via the bijection $\eta^*$. 
As $X^*((E^*)^{\times})= \ZZ[\Hom(E^*,\CC)]$, we denote by $\varepsilon_I$ the basis of $X^*((E^*)^{\times})$ corresponding to the coset $ \varepsilon_I H'$.

\begin{lem}\label{LemCMTypeAntiWeyl}
The reflex CM type $\Phi^*$ on $E^*$ is the subset of $\Hom (E^*,\CC)$ consisting of the cosets
$\{ \varepsilon_I H'\}$ with $I\subseteq\{1,\dots,g\}$ such that $1\notin I$ .
\end{lem}
\begin{proof}
By definition $\Phi^*$ is the set $\{  \gamma^{-1} H' : \gamma \in G \text{ with }\gamma H \in \Phi\}$ with $H$ from \eqref{EqStabPhi1}.  
By Lemma~\ref{LemmaLeftRightAction}, the condition $\gamma H \in \Phi$ is equivalent to $\gamma H = \sigma^{j-1} H$ for some $j$, and hence equivalent to $\gamma=(0,\sigma^{j-1} )(\varepsilon_I, h')$ for some $I$ not containing  $1$, some $h'\in S_{g-1}$ and some $j\in \{1,\dots,g\}$. In particular, $H\varepsilon_I \in \Phi$ if $1\not\in I$. 

On the other hand for each such $\gamma$, we have $\gamma^{-1} =  (\varepsilon_I , h^{\prime \! -1})(0,\sigma^{1-j})$ and hence $\gamma^{-1} H' = \varepsilon_I H'$ with $1 \not\in I$. Hence we are done.
\end{proof}

Let $A$ (resp. $A^*$) be a CM abelian variety associated with the CM pair $(E,\Phi)$ (resp. with $(E^*,\Phi^*)$).  Then $A$ is of Weyl type and $A^*$ is of anti-Weyl type. 
Let $T$ be the maximal torus of $\GSp(H^1(A,\QQ))$ described in $\mathsection$\ref{SubsectionAlgTori}. Then it is easy to check that $\mathrm{MT}(A)$ is $T$. 
Recall that $(E,\Phi)$ is primitive, so by the last paragraph of $\mathsection$\ref{SubsectionReflex}, 
the reflex of $(E^*, \Phi^*)$ is $(E,\Phi)$ and  $\mathrm{MT}(A^*) \cong \mathrm{MT}(A)$. So $\dim \mathrm{MT}(A^*)$ attains its lower bound $\log_2 \dim A^* +2$. This makes $A^*$ interesting from the point of view of Hodge theory.

We end this section with the following proposition on the reflex norm. Recall that we have identified $\Hom(E^*,\CC)$ with $G/H'$ via $\eta^*$ above. So $X^*((E^*)^{\times})=\ZZ[\Hom(E^*,\CC)]$ is identified with $\ZZ[G/H']$. 

\begin{prop}\label{PropReciprocityMap}
The group homomorphism induced by the reflex norm \eqref{EqReflexNorm} (but with $E$ and $E^*$ switched)
\[
\mathrm{rec}^*\colon X^*((E^*)^{\times})=\ZZ[G/H']\longrightarrow X^*(E^\times)=\ZZ[\Hom(E,\CC)]
\]
is given by: for all $I\subseteq \{1,\dots,g\}$, 
\[
\varepsilon_I H' \mapsto \sum_{j \not\in I} \phi_j + \sum_{j \in I} \overline{\phi}_j.
\]
 \end{prop}
\begin{proof}
The reflex norm \eqref{EqReflexNorm} applied to our situation is
\[
\mathrm{rec} \colon E^\times \rightarrow (E^*)^\times, \qquad a \mapsto \prod_{\phi \in \Phi} \phi(a).
\]
Since $\Phi = \{\phi_1,\ldots,\phi_g\}$, we have
\[
\mathrm{rec}^*( \varepsilon_I H' )= \sum_{j=0}^{g-1} \varepsilon_I \phi_j.
\]
Hence the result follows from \eqref{EqActionOfZ2gOnWeyl}.
\end{proof}

\section{Algebraic relations from the kernel for anti-Weyl CM abelian varieties}\label{SectionKernelReciprocity}

We continue to use the notation from $\S$\ref{s1.4}. Let $A^*$ be a CM abelian variety of anti-Weyl type, associated with $(E^*, \Phi^*)$. Let $T_o$ be its Mumford--Tate group.

Recall the short exact sequence in Lemma~\ref{LemNprimeReflex} (applied to $A^*$), obtained from $T_o < (E^*)^{\times}$ and the reflex norm
\[
0\longrightarrow N'\rightarrow X^*((E^*)^{\times})\xrightarrow{\mathrm{rec}^*} X^*(E^\times).
\]
 Non-zero elements of $N'$ give monomial relations between the holomorphic periods $\Theta_I$ (with $I\subseteq \{1,\dots,g\}$ not containing 1) and the anti-holomorphic periods $\Theta_I$ (with $1\in I$) and $\pi$.

The main result of this section is the following proposition.

\begin{prop} \label{PropKernelGeneratedByQuadratic}
The kernel $N'$ is generated by all the $\varepsilon_I+\varepsilon_J-\varepsilon_K-\varepsilon_L$ with $I\cap J = K\cap L$ and $I\cup J = K \cup L$.
\end{prop}

It has the following corollary.

\begin{cor}\label{CorThetaIGenerator}
For $I = \{i_1,\ldots,i_r\}$, we have
\[
\Theta_I = (\Theta_{\{i_1\}} \cdots \Theta_{\{i_r\}}) / \Theta_{\emptyset}^{r-1}.
\]
As a consequence, $\IQbar(\Theta_I)_{I \subseteq \{1,\ldots,g\}} = \IQbar(\Theta_{\emptyset}, \Theta_{\{1\}}, \ldots, \Theta_{\{g\}})$.
\end{cor}
\begin{proof}
We prove this corollary by induction on $r$. When $r=1$ the result trivially holds true. 

Assume the result for $r-1$. Then by Proposition~\ref{PropKernelGeneratedByQuadratic}, $\varepsilon_I + \varepsilon_{\emptyset} - \varepsilon_{I\setminus\{i_r\}} - \varepsilon_{\{i_r\}} \in N'$. Hence $\Theta_I \Theta_{\emptyset} = \Theta_{I\setminus \{i_r\}} \Theta_{\{i_r\}}$. Applying the induction hypothesis on $\Theta_{I\setminus\{i_r\}}$ yields the desired result.
\end{proof}

\medskip
The rest of this section is devoted to prove Proposition~\ref{PropKernelGeneratedByQuadratic}. We start with several lemmas.

\begin{lem}\label{Lemmagplus1generatingelements}
Let $M$ be the $g+1$-dimensional submodule  of $X^*((E^*)^{\times})$ defined by
\[
M=\ZZ \varepsilon_{\emptyset}\oplus \ZZ \varepsilon_{\{1\}}\oplus \dots \oplus \ZZ \varepsilon_{\{g\}}.
\]
Then  $X^*((E^*)^{\times})=M\oplus N' $.
\end{lem}
\begin{proof}
Since $\dim T_o = g+1$, it suffices to prove $M \cap N' = \{0\}$.

By Proposition~\ref{PropReciprocityMap}, we have
\[
\mathrm{rec}^*(k_0 \varepsilon_{\emptyset} + k_1 \varepsilon_{\{1\}} + \cdots + k_g \varepsilon_{\{g\}})  = \sum_{j=1}^g (k - k_j) \phi_j 
\]
where $k = k_0+k_1+\cdots+k_g$. Therefore 
$$
\mathrm{rec}^*(k_0 \varepsilon_{\{\emptyset\}} + k_1 \varepsilon_{\{1\}} + \cdots + k_g \varepsilon_{\{g\}})  = 0 \Rightarrow k_0 = k_1 = \ldots = k_g =0.
$$
 Hence $M \cap N' = \{0\}$ and we are done.
\end{proof}

From now we will denote by $\Sigma_g=\cP(\{1,\dots,g\})$ the set of subsets of $\{1,\dots,g\}$. For $I\in \Sigma_g$, we write $I^{\mathrm{c}}\in \Sigma_g$ for the complement of $I$, \textit{i.e.} 
\[
I\cap I^{\mathrm{c}}=\emptyset \mbox{ and } I\cup I^{\mathrm{c}}= \{1,\dots,g\}.
\]

The following lemma is an immediate corollary of Proposition~\ref{PropReciprocityMap}.
\begin{lem}\label{LemmaIIcKernel}
For all $I\in \Sigma_g$,
 the image of $\varepsilon_I+\varepsilon_{I^{\mathrm{c}}}$ under the reciprocity map $\mathrm{rec}^*$ is  $\sum_{\phi \in \Hom(E,\CC)}\phi$ and therefore is independent of $I$. 
\end{lem}

Let  $N_1$ be the submodule of $N'$ generated by all the $\varepsilon_I+\varepsilon_J-\varepsilon_K-\varepsilon_L$ with $I\cap J = K\cap L$ and $I\cup J = K \cup L$.

For any $e,f \in X^*((E^*)^{\times})$, we say that $e$ and $f$ are equivalent 
and we write $e\equiv f$ if $e-f\in M\oplus N_1$, where $M$ is from Lemma~\ref{Lemmagplus1generatingelements}. We will need the following lemma.

\begin{lem}\label{toto1}
For all subsets $I$, $I'$ such that $\vert I\vert=\vert I'\vert$, we have $\varepsilon_I\equiv \varepsilon_{I'}$.
\end{lem}
\begin{proof}
Let $I\in \Sigma_g$, let $j\notin I$, $k\in I$. Let $J=\{j\}$, $K=\{k\}$, $L=I\cup J \setminus K$. Then $\vert I\vert=\vert L\vert$,
$$
\varepsilon_I+\varepsilon_J-\varepsilon_K-\varepsilon_L\in N_1 \mbox{ and } \varepsilon_J-\varepsilon_K\in M.
$$
Hence $\varepsilon_I\equiv \varepsilon_L$. So $\varepsilon_I\equiv \varepsilon_{I'}$ if $\vert I\vert=\vert I'\vert$ and $\vert I\cap I'\vert =\vert I\vert-1$. 

Now take any  $I_1, I_2 \in \Sigma_g$ such that $|I_1| = |I_2|$. Let $O(I_1)$, resp. $O(I_2)$, be the equivalence class containing $I_1$, resp. $I_2$. 
Let $\lambda$ the be maximum of the $\vert I\cap J\vert$ for $I\in O(I_1)$ and $J\in O(I_2)$. We may assume that $\lambda=\vert I_1\cap I_2\vert$. If $O(I_1)\neq O(I_2)$,  then $\lambda< \vert I_1\vert$. Take $i_1\in I_1\setminus(I_1\cap I_2)$ and $i_2\in I_2\setminus(I_1\cap I_2)$. Set $I'_1=I_1 \cup \{i_2\} \setminus \{i_1\}$. Then $\varepsilon_{I'_1}\in O(I_1)$ by the conclusion of the last paragraph and $\vert I'_1\cap I_2\vert \ge \lambda+1>\lambda$. This contradiction the maximality of $\lambda$. Hence $O(I_1) = O(I_2)$ and we are done.
\end{proof}

Now we are ready to prove Proposition~\ref{PropKernelGeneratedByQuadratic}.

\begin{proof}[Proof of Proposition~\ref{PropKernelGeneratedByQuadratic}]
By Proposition~\ref{PropReciprocityMap}, each such $\varepsilon_I+\varepsilon_J-\varepsilon_K-\varepsilon_L$ is in $N'$.

Let $\langle \ ,\ \rangle$ be the scalar product on  $X^*((E^*)^{\times})$ with orthonormal  basis $\varepsilon_I$ for $I\in \Sigma_g$. Let $w$ be a vector in $(M \oplus N_1)^{\bot}$. By Lemma \ref{toto1} and the definition of $M$, we can write
\[
w=\sum_{j=2}^{g} c_j\sum_{\vert I\vert=j} \varepsilon_I.
\] 
Now $0 = \langle w, \varepsilon_{\{1,2\}} + \varepsilon_{\emptyset} - \varepsilon_{\{1\}} - \varepsilon_{\{2\}} \rangle = c_2$. Next $0 = \langle w, \varepsilon_{\{1,2,3\}} + \varepsilon_{\emptyset} - \varepsilon_{\{1,2\}} - \varepsilon_{\{3\}} \rangle = c_3$. 
Continuing this process we get $c_j=0$ for all $j \in \{2,\ldots,g\}$. Hence we are done.
\end{proof}

\section{Generalized anti-Weyl CM abelian varieties}\label{SectionAntiWeylGeneralisee}
The goal of this section is to introduce the notion of  \textit{generalized anti-Weyl CM abelian varieties}. In contrast to anti-Weyl CM abelian varieties, the generalized ones are not necessarily simple. Rather, we will show that each simple CM abelian variety is isogenous to an abelian subvariety of a generalized anti-Weyl CM abelian variety. Hence many problems about CM abelian varieties can be reduced to studying the generalized anti-Weyl ones, for example to study the Hodge relations between holomorphic periods. On the other hand,  many computations for anti-Weyl CM abelian varieties are still valid in this more general setting, for example the computation of the Hodge rings in the next section.

\vskip 0.3em

Let $E$ be a CM field of degree $2g$, and let $E^c$ be its Galois closure. Let $G = \mathrm{Gal}(E^c/\QQ)$, and let $\rho \in G$ be the complex conjugation. Fix a CM type $\Phi_E:=\{\phi_1,\ldots,\phi_g\}$ on $E$, with $\phi_1$ the inclusion $E \subseteq E^c \subseteq \CC$. Then  $\Hom(E,\CC) = \{\phi_1,\ldots,\phi_g, \overline{\phi}_1,\ldots,\overline{\phi}_g\}$.

By Dodson \cite[Imprimitivity Theorem]{DodsonThe-structure-o}, there is a natural inclusion 
$$G < (\ZZ/2\ZZ)^g \rtimes S_g,
$$ 
such that the image of $G \rightarrow S_g$ acts transitively on $\{1,\ldots,g\}$. Under this inclusion $\rho = (1,\ldots,1) \in (\ZZ/2\ZZ)^g$.

\subsection{Basic definitions}\label{SubsectionAntiWeylGenDefn}
Let $\{e_1,\ldots,e_g\}$ be the canonical basis of $(\ZZ/2\ZZ)^g$. We have  bijections 
\begin{equation}\label{EqBijectionsSubsetsZ2ZCMtypes}
\begin{array}{ccccc}
\cP(\{1,\ldots,g\}) & \cong & (\ZZ/2\ZZ)^g & \cong & \{\text{CM types of }E\}\\
I & \mapsto & \varepsilon_I:=\sum_{j\in I}e_j  & \mapsto & \{\phi_j : j\not\in I\} \cup \{\overline{\phi}_j : j \in I\}.
\end{array}
\end{equation}
The first bijection is natural and canonical, and the second  depends on the choice of the CM type $\Phi_E$. 
The group $G$ acts on each of the three sets via its natural action on $(\ZZ/2\ZZ)^g$. Here is the formula for the induced action on $\cP(\{1,\ldots,g\})$: For each $I \subseteq \{1,\ldots,g\}$ and each $\theta\in G$, we have
\begin{equation}\label{EqFormulaActionGOnI}
\theta\cdot I = \beta_{\theta}^{-1} (I \bigvee I_{\theta}) \qquad \text{for any }\theta = (\varepsilon_{I_{\theta}}, \beta_{\theta}) \in G < (\ZZ/2\ZZ)^g \rtimes S_g
\end{equation}
where $I \bigvee I_{\theta} = (I \cup I_{\theta}) \setminus (I \cap I_{\theta}) \subseteq \{1,\ldots,g\}$ and the right hand side is the usual action of $S_g$ on $\{1,\ldots,g\}$.  The induced action of $G$ on $\{\text{CM types of }E\}$ is  the one from \cite[$\S$1.2, Proposition]{DodsonThe-structure-o}.

Let $O_1,\ldots,O_m$ be the $G$-orbits, and up to renumbering assume $O_1$ is the orbit of $\emptyset$ (the orbit of $(0,\ldots,0)$, or the orbit of $\Phi_E$). 

We construct a CM abelian variety from the sets in \eqref{EqBijectionsSubsetsZ2ZCMtypes} as follows. Let
\begin{itemize}
\item $S := \cP(\{1,\ldots,g\})$ with the action of $G$ as above; 
\item $\mathbf{\Phi} := \{ I \subseteq \{1,\ldots, g\} : 1\not\in I\}$ and $\overline{\mathbf{\Phi}} := \{I \subseteq \{1,\ldots,g\} : 1\in I\}$.
\end{itemize}
Then $\Hom_G(S, E^c)$ is a CM algebra split by $E^c$, \textit{i.e.} it is isomorphic to the product of some CM fields in $E^c$. Indeed, we have an isomorphism
$$
\Hom_G(S, E^c)\simeq (E^c)^{\mathrm{Stab}_G(\Phi_1)} \times \cdots \times (E^c)^{\mathrm{Stab}_G(\Phi_m)},
$$
 with some $\Phi_1 \in O_1, \ldots, \Phi_m \in O_m$, and each $(E^c)^{\mathrm{Stab}_G(\Phi_r)}$ is the reflex field of $(E,\Phi_r)$ and hence is CM.

Note that $S = \Hom( \Hom_G(S,E^c), \CC)$. Hence $\mathbf{\Phi}$ defines a CM type on $\Hom_G(S, E^c)$. If we look at $S$ as the set of CM types on $E$ under the identification \eqref{EqBijectionsSubsetsZ2ZCMtypes}, then $\mathbf{\Phi}$ consists of those CM types which contains $\phi_1$.

Let $\cO$ be the ring of integers of $\Hom_G(S, E^c)$. 

\begin{defini}
A CM abelian variety is said to be {\em of generalized anti-Weyl type arising from $(E,\Phi_E)$}, if it is isogeneous to 
\[
\CC^{|S|/2} / \mathbf{\Phi}(\cO)
\]
for $S$, $\mathbf{\Phi}$, $\cO$ above.
\end{defini}
By definition, up to isogeny there exists a unique CM abelian variety of generalized anti-Weyl type arising from each CM pair $(E,\Phi_E)$, which we denote by $\mathcal{A}(E,\Phi_E)$. It is clear that $\dim \mathcal{A}(E,\Phi_E) = 2^{[E:\QQ]/2 -1}$.

\medskip
For each $r \in \{1,\ldots,m\}$, set $\mathbf{\Phi}_r := \mathbf{\Phi}\cap O_r$. As discussed above, each $\Hom_G(O_r, E^c)$ is a CM field, and we have a natural identification $O_r = \Hom(\Hom_G(O_r, E^c), \CC)$. Now $\mathbf{\Phi}_r$ is a CM type on $\Hom_G(O_r, E^c)$. Let $\cO_r$ be the ring of integers of 
$\Hom_G(O_r, E^c)$.  
Then by construction $\mathcal{A}(E,\Phi_E)$ is isogeneous to
\begin{equation}\label{EqAntiWeylGenDecom}
\CC^{|O_1|/2}/\mathbf{\Phi}_1(\cO_1) \times\cdots \times \CC^{|O_m|/2}/\mathbf{\Phi}_m(\cO_m).
\end{equation}
Each $\CC^{|O_r|/2}/\mathbf{\Phi}_r(\cO_r)$ is a CM abelian variety associated with the CM pair $(\Hom_G(O_r, E^c), \mathbf{\Phi}_r)$. 

As above, up to renumbering we always assume $O_1$ to be the orbit of $\emptyset$ (the orbit of $(0,\ldots,0)$ or the orbit of $\Phi_E$ under the identification \eqref{EqBijectionsSubsetsZ2ZCMtypes}).

\begin{defini}
Let $A$ be a simple CM abelian variety. Let $(E,\Phi_E)$ be the reflex CM pair of the  CM pair associated with $A$ and $\cA(E,\Phi_E)$ be the associated generalized anti-Weyl abelian variety. We call the abelian varieties $\CC^{|O_1|/2}/\mathbf{\Phi}_1(\cO_1), \ldots, \CC^{|O_m|/2}/\mathbf{\Phi}_m(\cO_m)$ from \eqref{EqAntiWeylGenDecom} the {\em compagnons of $A$}.
\end{defini}

A main result of this section is to prove that  the first compagnon is precisely $A$ up to isogeny (Proposition~\ref{PropGenAntiWeylFirstFactor}). 
An a consequence, any simple CM abelian variety is, up to isogeny, an abelian subvariety of a CM abelian variety of generalized anti-Weyl type.

\subsection{The Galois action on $\Hom(E,\CC)$}

Use the notation from $\mathsection$\ref{SubsectionAntiWeylGenDefn}. 
We fix some further notation for the computation. Let
 $$
 H:= \mathrm{Fix}(\phi_1) = \mathrm{Gal}(E^c/E) = G \cap (\{0\} \times (\ZZ/2\ZZ)^{g-1}\rtimes S_{g-1}).
 $$
  Then $|G/H| = [E:\QQ] = 2g$. 

For each $j \in \{1,\ldots,g\}$, there exists $\alpha_j = (\varepsilon_{I_j}, \beta_j) \in G < (\ZZ/2\ZZ)^g \rtimes S_g$ such that $\beta_j(1) = j$. This is because  the image of $G<(\ZZ/2\ZZ)^g \rtimes S_g \rightarrow S_g$ acts transitively on $\{1,\ldots,g\}$. Up to replacing $\alpha_j$ by $\rho \circ \alpha_j$, we may and do assume that $j\not\in I_j$ for each $j\in \{1,\ldots,g\}$. Now
\[
G/H = \{\alpha_1 H, \ldots, \alpha_g H, \alpha_1 \rho H, \ldots, \alpha_g \rho H\}
\]
since $|G/H| = 2g$. We can define a bijection
\begin{equation}\label{EqEEmbeddingRepre}
\eta \colon \Hom(E,\CC) \xrightarrow{\sim} G/H
\end{equation}
by sending $\phi_j \mapsto \alpha_j H$ and $\overline{\phi}_j\mapsto \alpha_j \rho H$ for each $j \in \{1,\ldots,g\}$.


\begin{rem}
By definition, $H$ fixes $\phi_1$ but not necessarily $\Phi$ for any CM type $\Phi$, even for $\Phi = \Phi_E$. Nevertheless, each $h \in H$ defines an isogeny $E_{\mathbb{R}}/\Phi(\mathcal{O}_E) \rightarrow E_{\mathbb{R}}/(h\Phi)(\mathcal{O}_E)$.
\end{rem}

As for the case of Weyl type CM abelian varieties, we prove the following lemma. Notice that this gives a new proof of Lemma~\ref{LemmaLeftRightAction}.
\begin{lem}\label{LemmaLeftRightActionAntiWeyl}
The bijection $\eta$ in $G$-equivariant, for the natural action of $G$ on $\Hom(E,\CC)$ and the action of $G$ on $G/H$ by multiplication on the left.
\end{lem}
\begin{proof}
Let $\theta = (\varepsilon_{I_{\theta}},\beta_{\theta}) \in G < (\ZZ/2\ZZ)^g\rtimes S_g$. For each $j \in \{1,\ldots,g\}$, set $j' := \beta_{\theta}(j)$. 

On the one hand, $\theta \alpha_j = (\varepsilon_{I_{\theta}} + \varepsilon_{\beta_{\theta}(I_j)}, \beta_{\theta}\beta_j)$. Since $\beta_{\theta}\beta_j(1) = j'$, we have that $\theta \alpha_j H = \alpha_{j'}H$ or $\theta \alpha_j H  = \alpha_{j'}\rho H$. Moreover, $\theta \alpha_j H = \alpha_{j'}H$ if and only if $j' \in I_{\theta} \bigvee \beta_{\theta}(I_j)$. Since $j \not\in I_j$, we have $j' = \beta_{\theta}(j) \not\in \beta_{\theta}(I_j)$. So  $\theta \alpha_j H = \alpha_{j'}H$ if and only if $j' \not\in I_{\theta}$.

On the other hand, $\theta \cdot \phi_j = \phi_{j'}$ or $\theta \cdot \phi_j = \overline{\phi}_{j'}$, with $\theta \cdot \phi_j = \phi_{j'}$ if and only if $j' \not\in I_{\theta}$.

This proves the $G$-equivariance for the $\phi_j$'s.  A similar computation allows us to conclude also for the  $G$-equivariance for the $\overline{\phi}_j$'s. This finishes the proof.
\end{proof}

\subsection{The first Galois orbit $O_1$} 
Recall that $O_1$ is the $G$-orbit of $\emptyset$ in $\cP(\{1,\ldots,g\})$; under the identifications \eqref{EqBijectionsSubsetsZ2ZCMtypes} it is the orbit of $(0,\ldots,0)$ in $(\ZZ/2\ZZ)^g$, or the orbit of $\Phi_E$ in $\{\text{CM types of }E\}$.
\begin{prop}\label{PropGenAntiWeylFirstFactor}
 $\CC^{|O_1|/2}/\mathbf{\Phi}_1(\cO_1)$ is the CM abelian variety associated with the reflex of $(E,\Phi_E)$.
\end{prop}
\begin{proof}
Denote for simplicity $H' = \mathrm{Stab}_G(\Phi_E)$. Since $\Phi_E$ corresponds to
 $$
 \mathbf{0} = (0,\ldots,0) \in (\ZZ/2\ZZ)^g
 $$
  under \eqref{EqBijectionsSubsetsZ2ZCMtypes}, we have $H' =  \mathrm{Stab}_G(\mathbf{0}) = G \cap (\{0\}\times S_g)$. 

We have $\Hom_G(O_1,E^c) \cong (E^c)^{H'}$ and as
 $$
 \Hom(\Hom_G(O_1,E^c),\CC)  = O_1 \cong G/H',
 $$
  the reflex CM type $\Phi_E^*$ equals  $\{ \theta^{-1} H' : \theta \in G \text { with } \theta H \in \Phi_E\}$. We want to prove that 
  $$
  \Phi_E^* = \mathbf{\Phi}_1 \subseteq O_1.
  $$

By Lemma~\ref{LemmaLeftRightAction}, 
the condition $\theta H \in \Phi_E$ is equivalent to $\theta H = \alpha_j H$ for some $j$, and hence equivalent to $\theta = (\varepsilon_{I_j}, \beta_j)(\varepsilon_I , h')$ for some $I$ not containing $1$, some $h' \in S_{g-1}$ such that $(\varepsilon_I, h') \in G$, and some $j$. Then $\theta^{-1} = (\varepsilon_{h^{\prime -1}(I)}, h^{\prime -1}) (\varepsilon_{\beta^{-1}_j(I_j)}, \beta_j^{-1}) = (\varepsilon_{h^{\prime -1}(I)} + \varepsilon_{h^{\prime -1}\beta^{-1}_j(I_j)} , h^{\prime -1}\beta_j^{-1})$. Let $I_{\theta^{-1}} \subseteq \{1,\ldots, g\}$ be such that $\varepsilon_{I_{\theta^{-1}}} = \varepsilon_{h^{\prime -1}(I)} + \varepsilon_{h^{\prime -1}\beta^{-1}_j(I_j)}$. Then $1 \not\in I_{\theta^{-1}}$ because $1 \not \in I$, $1 \not \in \beta^{-1}_j(I_j)$ and $h' \in S_{g-1}$. 

In summary, the reflex norm $\Phi_E^*$ is contained in $\mathbf{\Phi}_1 = \mathbf{\Phi} \cap O_1 \subseteq O_1 \cong G/H'$. 
 But $|\Phi_E^*| = |\mathbf{\Phi}_1| = |O_1|/2$. Hence $\Phi_E^* = \mathbf{\Phi}_1$.  We are done.
\end{proof}

Here are two immediate corollaries.
\begin{cor}\label{CorSimpleInAntiWeylGen}
Each simple CM abelian variety is a compagnon of itself.
\end{cor}

\begin{cor}
Each anti-Weyl type CM abelian variety is of generalized anti-Weyl type.
\end{cor}

\subsection{Powers of CM abelian varieties of generalized anti-Weyl type}\label{SubsectionPowerNotation} 

Let \\   $A:=\cA(E,\Phi_E)$ be an abelian variety of generalized anti-Weyl type arising from $(E,\Phi_E)$. Later on we will study an arbitrary power of $\cA(E,\Phi_E)$. In this subsection, we make a preliminary discussion and we fix some notations. 

Recall that $A$ is associated with the CM pair $\left( \Hom_G(S,E^c), \mathbf{\Phi} \right)$, with $S = \cP(\{1,\ldots,g\})$ and $\mathbf{\Phi} = \{I \subseteq \{1,\ldots,g\} : 1\not\in I\}$. Using the notation of $\S$\ref{SubsectionCMType},  for each $n \ge 1$, we can construct a CM pair as follows. The CM algebra is $ \Hom_G(S,E^c)^n$, with $\Hom ( \Hom_G(S,E^c)^n, \CC) = S^{\sqcup n}$; the CM type is $\mathbf{\Phi}^{\sqcup n}$.

The CM algebra $\Hom_G(S,E^c)$ acts on the space of holomorphic $1$-forms (resp. anti-holomorphic $1$-forms) on $A$ via $\mathbf{\Phi}$ (resp. via $\overline{\mathbf{\Phi}}$), and hence we have an eigenbasis
\[
\{\varepsilon_I\}_{I\subseteq \{1,\dots,g\}, 1\notin I} , \quad \text{(resp. }\{\varepsilon_I\}_{I\subseteq \{1,\dots,g\}, 1\in I}\text{)}
\]
for this action, with $\varepsilon_I$ an eigenvector associated with $I \in \mathbf{\Phi}$ (resp. with $I \in \overline{\mathbf{\Phi}}$).

For each $\ell \in \{1,\ldots,n\}$, let
\[
I^{(\ell)}, \quad \text{(resp. }\varepsilon_I^{(\ell)} \text{)}
\]
denote the subset $I \in \cP(\{1,\ldots,g\})$ on the $\ell$-th copy of $\cP(\{1,\ldots,g\})^{\sqcup n}$ (resp. the $1$-form $\varepsilon_I$ on the $\ell$-th factor of $A^n$).

We fix a partial order on $S = \cP(\{1,\ldots,g\})$ such that $I \le J$ if $1\notin I$ and $1\in J$ and $I \le J \Rightarrow J^{\mathrm{c}} \le I^{\mathrm{c}}$. Then it induces a partial order on $S^{\sqcup n}$ such that $1\le \ell_1\le \ell_2 \le n \Rightarrow I^{(\ell_1)} \le J^{(\ell_2)}$.

\section{Hodge cycles on generalized anti-Weyl CM abelian varieties}\label{SectionHodgeCyclesAntiWeylGen}

Let $A$ be a generalized anti-Weyl CM abelian variety arising from $(E,\Phi_E)$, with $E$ a CM field of degree $2g$. Then $\dim A = 2^{g-1}$. Let $G = \mathrm{Gal}(E^c/\QQ)$. Then
\[
G < (\ZZ/2\ZZ)^g \rtimes S_g, \quad \text{and  the image of }G \rightarrow S_g\text{ acts transitively on }\{1,\ldots,g\}.
\]

The purpose of this section is characterize the structure of the Hodge ring $B(A^{n})$ for all $n \ge 1$ and the Hodge relations between the  periods of $A$. Retain the notation in $\S$\ref{SubsectionPowerNotation}. We have the following results.

\begin{teo}\label{TheoremHodge22AntiWeyl}
Let $B^2(A^{n}):= H^{2,2}(A^n ,\CC)\cap H^4(A^n ,\QQ)$. Then $B^2(A^{n})\otimes \CC$ has a basis consisting of non-zero vectors of the form
$$
\varepsilon_I^{(\ell_1)}\wedge \varepsilon_J^{(\ell_2)}\wedge \varepsilon_{K^{\mathrm{c}}}^{(\ell_3)}\wedge \varepsilon_{L^{\mathrm{c}}}^{(\ell_4)}
$$
for any $\ell_1, \ell_2, \ell_3, \ell_4 \in \{1,\ldots,n\}$, 
and $I,J,K,L$ subsets of $\{2,\dots,g\}$  such that $I,J,K^{\mathrm{c}}, L^{\mathrm{c}}$ is ordered and
\begin{equation}\label{EqConditionIJKL}
I\cup J= K\cup L \mbox{ and }  I\cap J=K\cap L. 
\end{equation}
In particular, the dimension of $B^2(A^{n})$ as a $\QQ$-vector space is the number of quadruples $(I,J,K,L)$, with $I\le J$ and $K^{\mathrm{c}}\le L^{\mathrm{c}}$ satisfying \eqref{EqConditionIJKL}.
\end{teo}
By \eqref{EqConditionIJKL}, we have two possibilities: (1) $I, J, K, L$ are two-by-two different; (2) $I=L$ and $J=K$. Notice that in the second possibility, it may happen that $I=J=K=L$ and then $\ell_1\not= \ell_2$ and $\ell_3\not=\ell_4$. 

We shall see that the indices $\ell_1,\ell_2,\ell_3,\ell_4$ are irrelevant when discussing  Hodge relations. Indeed, for a fixed $I \subseteq \{1,\ldots,g\}$, the non-zero integral of $\varepsilon_I^{(\ell)}$ is independent of the choice of $\ell \in \{1,\ldots,n\}$ and is a period of $A$, which we denote by $\Theta_I$. Those with $1 \not\in I$ (resp. with $1 \in I$) are holomorphic (resp. anti-holomorphic). An immediate consequence of Theorem~\ref{TheoremHodge22AntiWeyl} is that the Hodge relations of degree $2$ between the holomorphic periods of $A$ are 
\begin{equation}\label{EqQuadraticHodgeRelation}
\Theta_I\Theta_J = \Theta_K\Theta_L
\end{equation}
for all $I, J , K, L \subseteq \{2,\ldots,g\}$ satisfying \eqref{EqConditionIJKL}. They can be all obtained by $(2,2)$-Hodge cycles on $A$.

\begin{teo}\label{ThmHodgeRingAntiWeyl}
The Hodge relations between the periods of $A^n$ are generated by those induced by $(1,1)$- and $(2,2)$-Hodge cycles on $A$.
\end{teo}
Moreover, any Hodge relation in degree $1$ is of the form $\Theta_I \Theta_{I^{\mathrm{c}}} = 2\pi i$ by Corollary~\ref{Cor11CycleAntiWeylGen}.

Our proof of Theorem~\ref{ThmHodgeRingAntiWeyl} is inspired by the study of the kernel of the reflex norm for anti-Weyl CM abelian varieties in $\mathsection$\ref{SectionKernelReciprocity}, and in particular Corollary~\ref{CorThetaIGenerator}.

\subsection{Proof of our main theorem}\label{SubsectionProofMainThm}
We start by explaining how to deduce Theorem~\ref{mainthm} from Theorem~\ref{ThmHodgeRingAntiWeyl} . The notations are  different from the rest of this section.

Let $A$ be a simple CM abelian variety, associated with the CM pair $(E,\Phi)$. Its reflex $(E^*, \Phi^*)$ is a CM pair with $E^*$ a field. Let $\cA(E^*,\Phi^*)$ be the generalized anti-Weyl CM abelian variety arising from $(E^*,\Phi^*)$. By Corollary~\ref{CorSimpleInAntiWeylGen}, up to isogeny $\cA(E^*,\Phi^*) = A \times B$ for some CM abelian variety $B$. By construction of $\cA(E^*,\Phi^*)$, we see that $B$ is split over $E^c$. Theorem~\ref{mainthm} then follows from Theorem~\ref{ThmHodgeRingAntiWeyl} and the comment below (that Hodge relations in degree $1$ do not give algebraic relations between the holomorphic periods).

\subsection{A first discussion on Hodge cycles on generalized anti-Weyl CM abelian varieties}

\begin{prop}\label{PropHdgRingAntiWeylPrem}
Let $B^p(A^{n}):=H^{p,p}(A^n ,\CC)\cap H^{2p}(A^n ,\QQ)$ be the $\QQ$-vector space of  $\QQ$-Hodge cycles of type $(p,p)$ on $A^n $. Then $B^p(A^{n}) \otimes \CC$ has a basis consisting of ordered $I_1^{(\ell_1)} \le \ldots \le I_{2p}^{(\ell_{2p})}$ such that each $j \in \{1,\ldots g\}$ appears exactly $p$ times in $I_1^{(\ell_1)},\ldots,I_{2p}^{(\ell_{2p})}$.
\end{prop}
\begin{proof}
Recall $S = \cP(\{1,\ldots,g\})$. 
We wish to apply Theorem~\ref{Poh} (with $S$ being our $S^{\sqcup n}$ and $\Phi$ being our $\mathbf{\Phi}^{\sqcup n}$). Thus the set $P$  from Theorem~\ref{Poh} is in our case of the form $\{I_1^{(\ell_1)},\ldots,I_{2p}^{(\ell_{2p})}\}$ with $I_j^{(\ell_j)} \in \cP(\{1,\ldots,g\})$. 

Recall that $\mathbf{\Phi}  = \{I \in \cP(\{1,\ldots,g\}) : 1 \not\in I\}$ and $\overline{\mathbf{\Phi} } = \{I \in \cP(\{1,\ldots,g\}) : 1 \in I\}$. Hence $\vert  P\cap \mathbf{\Phi}^{\sqcup n} \vert=\vert  P\cap \overline{\mathbf{\Phi} }^{\sqcup n}\vert$ is equivalent to  $1$ being contained in exactly half of the sets $I_1^{(\ell_1)},\ldots,I_{2p}^{(\ell_{2p})}$. 

The condition given by the equation \eqref{toto} is in our situation $\vert  P\cap \sigma\mathbf{\Phi}^{\sqcup n} \vert=\vert  P\cap \sigma\overline{\mathbf{\Phi} }^{\sqcup n}\vert$ for all $\sigma \in G < (\ZZ/2\ZZ)^2 \rtimes S_g$. 
For the canonical basis $\{e_1,\ldots,e_g\}$ of $(\ZZ/2\ZZ)^g$, we have  $e_1\cdot \mathbf{\Phi}^{\sqcup n} = \overline{\mathbf{\Phi}}^{\sqcup n}$ and $e_1 \cdot \overline{\mathbf{\Phi}}^{\sqcup n} = \mathbf{\Phi}^{\sqcup n}$, and $e_j \cdot \mathbf{\Phi}^{\sqcup n}  = \mathbf{\Phi}^{\sqcup n} $ and $e_j \cdot \overline{\mathbf{\Phi} }^{\sqcup n} = \overline{\mathbf{\Phi} }^{\sqcup n}$ for all $j \ge 2$. Hence \eqref{toto} is further equivalent to $\vert  P\cap \sigma\mathbf{\Phi}^{\sqcup n} \vert=\vert  P\cap \sigma\overline{\mathbf{\Phi} }^{\sqcup n}\vert$ for all $\sigma \in \im(G \rightarrow S_g)$. To ease notation, for each $j \in \{1,\ldots, g\}$, set $\mathbf{\Phi}_j := \{I \subseteq \{1,\ldots, g\} : j \not\in I\}$ and $\overline{\mathbf{\Phi}_j} := \{I \subseteq \{1,\ldots, g\} : j \in I\}$.  We have
\[
\{\sigma\mathbf{\Phi}^{\sqcup n}  : \sigma \in G \} \cong \{\mathbf{\Phi}_j^{\sqcup n} : j =1, \ldots,g \} \quad \text{and} \quad \{\sigma\overline{\mathbf{\Phi} }^{\sqcup n} : \sigma \in G\} \cong \{\overline{\mathbf{\Phi}_j}^{\sqcup n} : j = 1,\ldots, g\}
\]
as $\im(G \rightarrow S_g)$ acts transitively on $\{1,\ldots,g\}$. 
Therefore the condition \eqref{toto} is equivalent to: $\vert  P\cap \mathbf{\Phi}_j^{\sqcup n}\vert=\vert  P\cap \overline{\mathbf{\Phi}_j}^{\sqcup n}\vert$ for all $j \in \{1,\ldots,g\}$. As for the case $j=1$ discussed above, this  is equivalent to: each $j \in \{1,\ldots,g\}$ is contained in exactly half of the sets $I_1^{(\ell_1)},\ldots,I_{2p}^{(\ell_{2p})}$. Hence we are done.
\end{proof}

Proposition~\ref{PropHdgRingAntiWeylPrem} yields immediately the following corollary.
\begin{cor}\label{Cor11CycleAntiWeylGen}
For each $n \ge 1$, $B^1(A^n) \otimes \CC$ has a basis consisting of vectors of the form $\varepsilon_I^{(\ell_1)} \wedge \varepsilon_{I^{\mathrm{c}}}^{(\ell_2)}$ for $I\subseteq \{2,\ldots,g\}$, with $\ell_1, \ell_2 \in \{1,\ldots,n\}$. 
\end{cor}

\subsection{$(2,2)$-Hodge cycles on generalized anti-Weyl CM abelian varieties}
As indicated by Theorem~\ref{ThmHodgeRingAntiWeyl}, $(2,2)$-Hodge cycles play a particularly important role in the Hodge rings of generalized anti-Weyl CM abelian varieties. It is therefore important to understand the Galois action on these $(2,2)$-Hodge cycles. We will do this in the current subsection. We will then prove Theorem~\ref{TheoremHodge22AntiWeyl} from the description of this Galois action. Notice that one can also prove Theorem~\ref{TheoremHodge22AntiWeyl} directly via Proposition~\ref{PropHdgRingAntiWeylPrem} without analysing the Galois action on the $(2,2)$-cycles.

Recall the formula for the action of $G$ on $\cP(\{1,\ldots,g\})$ from \eqref{EqFormulaActionGOnI}: 
\[
\theta \cdot M = \beta_{\theta}^{-1} (M \bigvee I_{\theta}) \qquad \text{for any }M\subseteq \{1,\ldots,g\}\text{ and any }\theta = (\varepsilon_{I_{\theta}}, \beta_{\theta}) \in G < (\ZZ/2\ZZ)^g \rtimes S_g,
\]
with  $M \bigvee I_{\theta} = (M \cup I_{\theta}) \setminus (M \cap I_{\theta}) \subseteq \{1,\ldots,g\}$ and the right hand side is the usual action of $S_g$ on $\{1,\ldots,g\}$. Notice that this action extends to an action of the whole $(\ZZ/2\ZZ)^g\rtimes S_g$ on $\cP(\{1,\ldots,g\})$ by the same formula.

\medskip
Here is the Galois action on the $(2,2)$-forms. 
For each $\theta \in G$ and any $I, J, K, L \subseteq \{2,\ldots,g\}$
, we have
\begin{equation}
\theta(\varepsilon_I^{(\ell_1)} \wedge \varepsilon_J^{(\ell_2)} \wedge \varepsilon_{K^{\mathrm{c}}}^{(\ell_3)} \wedge \varepsilon_{L^{\mathrm{c}}}^{(\ell_4)}) = \varepsilon_{\theta\cdot I}^{(\ell_1)} \wedge \varepsilon_{\theta\cdot J}^{(\ell_2)} \wedge \varepsilon_{\theta\cdot K^{\mathrm{c}}}^{(\ell_3)} \wedge \varepsilon_{\theta\cdot L^{\mathrm{c}}}^{(\ell_4)}.
\end{equation}

\medskip
We start with the following lemma. The computation involved in its proof is useful.
\begin{lem}\label{LemmaIJKLtheta}
Let $I, J, K,L$ be subsets of $\{1,\dots,g\}$ such that 
\[
I\cup J=K\cup L \mbox{ and }  I\cap J=K\cap L.
\]
  Then
\[
(\theta\cdot I)\cup (\theta\cdot J)=(\theta\cdot K)\cup (\theta\cdot L) \mbox{ and }  (\theta\cdot I)\cap (\theta\cdot J)=(\theta\cdot K)\cap (\theta\cdot L)
\]
for all $\theta \in (\ZZ/2\ZZ)^g \rtimes S_g$ (and in particular for all $\theta \in G$).
\end{lem}

\begin{proof}
As $\theta = \left(\prod_{j\in I_{\theta}}(\varepsilon_{\{j\}},1)\right)(0,\beta_{\theta})$, it suffices to prove the result for $(0,\beta_{\theta})$ and for $e_j=\varepsilon_{\{j\}}$ for each $j\in \{1,\ldots,g\}$.

The expected conclusion easily holds true for $(0,\beta_{\theta})$, which gives a bijection on $\{1,\ldots,g\}$. Hence it remains to check for $e_j=\varepsilon_{\{j\}}$ for each $j\in \{1,\dots,g\}$. There are 3 cases up to the symmetries. 
\begin{itemize} \label{abc}
\item (a) $j\in I$, $j\notin J$, $j\in K$, $j\notin L$.
\item (b) $j\in I\cap J=K\cap L$.
\item (c) $j\in (I\cup J)^{\mathrm{c}}=(K\cup L)^{\mathrm{c}}$.
\end{itemize}

In case (a), we have 
\begin{equation}\label{EqCaseA}
e_j\cdot I = I\setminus\{j\}, ~ e_j\cdot J=J\cup \{j\}, ~ e_j\cdot K=K\setminus\{j\}, ~ e_j\cdot L=L\cup\{j\}.
\end{equation}
Hence $(e_j\cdot I)\cup (e_j\cdot J) = I\cup J = K\cup L = (e_j\cdot K)\cup (e_j\cdot L)$ and $(e_j\cdot I) \cap (e_j\cdot J) = I\cap J = K\cap L = (e_j\cdot K)\cap (e_j\cdot L)$.

\medskip
In case (b), we have
\begin{equation}\label{EqCaseB}
(e_j\cdot I)=I\setminus\{j\}, ~e_j\cdot J=J\setminus\{j\},~ e_j\cdot K=K\setminus\{j\}, ~e_j\cdot L=L\setminus\{j\}.
\end{equation}
 Therefore 
$
(e_j\cdot I)\cup (e_j\cdot J)= (I\cup J)\setminus\{j\}=(K\cup L)\setminus\{j\}=(e_j\cdot K)\cup (e_j\cdot L)
$
and 
$
(e_j\cdot I)\cap (e_j\cdot J)= (I\cap J)\setminus\{j\}=(K\cap L)\setminus\{j\}=(e_j\cdot K)\cap (e_j\cdot L)
$.

\medskip

In case (c), we have
\begin{equation}\label{EqCaseC}
e_j\cdot I=I\cup \{j\}, ~e_j\cdot J=J\cup \{j\}, ~e_j\cdot K=K\cup \{j\}, ~e_j\cdot L=L\cup \{j\}.
\end{equation}
 Therefore
$
(e_j\cdot I)\cup (e_j\cdot J)= (I\cup J)\cup \{j\}=(K\cup L)\cup \{j\}=(e_j\cdot K)\cup (e_j\cdot L)
$
and 
$
(e_j\cdot I)\cap (e_j\cdot J)= (I\cap J)\cup \{j\}=(K\cap L)\cup \{j\}=(e_j\cdot K)\cap (e_j\cdot L)
$.
\end{proof}

The key point to prove Theorem~\ref{TheoremHodge22AntiWeyl} is the following lemma.
\begin{lem}\label{lem2.4}
Let $I,J,K, L$ be subsets of $\{2,\dots,g\}$. 
\begin{enumerate}
\item[(i)]  If $I\cup J=K\cup L$ and $I\cap J=K\cap L$, then $1$ is contained in exactly 2 of the sets $(\theta\cdot I), (\theta\cdot J),(\theta\cdot K^{\mathrm{c}}),(\theta\cdot L^{\mathrm{c}})$ for all $\theta \in (\ZZ/2\ZZ)^g \rtimes S_g$ (and in particular for all $\theta \in G$).
\item[(ii)] 
If $I\cup J\neq K\cup L$ or  $I\cap J\neq K\cap L$, then there exists $\theta \in G$ such that  $1$ is contained in $\ge 3$ of the sets $(\theta\cdot I), (\theta\cdot J),(\theta\cdot K^{\mathrm{c}}),(\theta\cdot L^{\mathrm{c}})$.
\end{enumerate}
\end{lem}

\begin{proof}
We start with the proof of (i). Since $\theta = \left(\prod_{j\in I_{\theta}}(\varepsilon_{\{j\}},1)\right)(0,\beta_{\theta})$, it suffices to prove the result for $(0,\beta_{\theta})$ and for $e_j$ for each $j\in \{1,\ldots,g\}$. 

Start with the case $\theta = (0,\beta_{\theta})$. Set $j = \beta_{\theta}(1) \in \{1,\ldots,g\}$.  Then by definition $\theta \cdot \{j\} = \beta_{\theta}^{-1}(\{j\}) = \{1\}$.   We may suppose that we are in one of the 3 cases described in the proof of Lemma~\ref{LemmaIJKLtheta}.

In case (a), $j \in I$, $j \notin  J$, $j\notin K^{\mathrm{c}}$, $j \in L^{\mathrm{c}}$. Applying $\theta$, we get $1\in (\theta\cdot I) $, $1\notin (\theta\cdot J)$, $1\notin (\theta\cdot K^{\mathrm{c}})$ and $1\in (\theta\cdot L^{\mathrm{c}})$.

In case (b), $j\in I$, $j \in J$, $j \notin K^{\mathrm{c}}$, $j \not\in L^{\mathrm{c}}$. Applying $\theta$, we get $1\in (\theta\cdot I) $, $1\in (\theta\cdot J)$, $1\notin (\theta\cdot K^{\mathrm{c}})$ and $1\notin (\theta\cdot L^{\mathrm{c}})$.

In case (c), $j \not\in I$, $j\not\in J$, $j \in K^{\mathrm{c}}$, $j \in L^{\mathrm{c}}$. Applying $\theta$, we get  $1\notin (\theta\cdot I)$, $1\notin (\theta\cdot J)$, $1\in (\theta\cdot K)^c$ and $1\in (\theta\cdot L)^c$.

\smallskip

So we just need to check (i) with $\theta=e_j$ for all $j \in \{1,\ldots,g\}$.

If $j\neq 1$, then $\theta\cdot I=I\cup \{j\}$ or $\theta\cdot I=I\setminus\{j\}$, and hence $1\notin \theta\cdot I$. In the same way $1\notin \theta\cdot J$, $1\in \theta\cdot K^{\mathrm{c}}$ and $1\in \theta\cdot L^{\mathrm{c}}$.

If $j=1$ then $\theta\cdot I=I\cup \{1\}$ so $1\in \theta\cdot I$. In the same way $1\in \theta\cdot J$, $1\notin \theta\cdot K^{\mathrm{c}}$ and $1\notin \theta\cdot L^{\mathrm{c}}$.

This finishes the proof of (i).

\medskip

Now let us prove (ii). Assume $I\cup J\neq K\cup L$. Let $j\in I\cup J$ such that $j\notin K\cup L$. We may assume that $j\in I$. Let $\theta = (\varepsilon_{I_{\theta}},\beta_{\theta}) \in G$ be such that $\beta_{\theta}(1) = j$; such a $\theta$ exists since the image of $G \rightarrow S_g$ acts transitively on $\{1,\ldots,g\}$. Moreover up to replacing $\theta$ by $\theta \rho$, we may and so assume that $j \not\in I_{\theta}$. Hence $j \in I_{\theta}\bigvee \{j\}$ and therefore $1 \in \theta\cdot \{j\} = \beta_{\theta}^{-1}(I_{\theta}\bigvee \{j\})$. Therefore  $1\in \theta\cdot I$ and $1\in \theta\cdot K^{\mathrm{c}}$ and $1\in \theta\cdot L^{\mathrm{c}}$. So $1$ is in at least 3 of the sets $(\theta\cdot I),(\theta\cdot J),(\theta\cdot K), (\theta\cdot L)$. So we are done for this case.

Assume  $I\cup J=K\cup L$ but $I\cap J\neq K\cap L$. Then there exists $j\in I\cap J$ such that $j\notin K\cap L$. We may assume $j\notin K$. Similarly to the previous case, there exists $\theta = (\varepsilon_{I_{\theta}},\beta_{\theta}) \in G$  such that $\beta_{\theta}(1) = j$ and $j \not\in I_{\theta}$. Then  $1 \in \theta\cdot \{j\} = \beta_{\theta}^{-1}(I_{\theta}\bigvee \{j\})$ as before. 
So $1\in \theta\cdot I$, $1\in \theta\cdot J$ and $1\in \theta\cdot K^{\mathrm{c}}$. So $1$ is at least in 3 of the sets $(\theta\cdot I),(\theta\cdot J),(\theta\cdot K),(\theta\cdot L)$. Now we are done.
\end{proof}

Now we are ready to finish the proof of Theorem~\ref{TheoremHodge22AntiWeyl}. 
\begin{proof}[Proof of Theorem~\ref{TheoremHodge22AntiWeyl}]
By Pohlmann's Theorem~\ref{Poh},  $B^2(A^n)\otimes \CC$ has a basis consisting of elements of the form $\varepsilon_I^{(\ell_1)}\wedge \varepsilon_J^{(\ell_2)}\wedge \varepsilon_{K^{\mathrm{c}}}^{(\ell_3)}\wedge \varepsilon_{L^{\mathrm{c}}}^{(\ell_4)}$ with $I,J,K,L \subseteq \{2,\ldots,g\}$ such that for all $\theta\in G$, we have the following property:  
the form
\[
\varepsilon_{\theta\cdot I}^{(\ell_1)} \wedge \varepsilon_{\theta\cdot J}^{(\ell_2)} \wedge \varepsilon_{\theta\cdot K^{\mathrm{c}}}^{(\ell_3)} \wedge \varepsilon_{\theta\cdot L^{\mathrm{c}}}^{(\ell_4)}
\]
is in $B^2(A^n)\otimes \CC$, \textit{i.e.} $1$ appears in exactly $2$ of the sets $(\theta\cdot I), (\theta\cdot J),(\theta\cdot K^{\mathrm{c}}),(\theta\cdot L^{\mathrm{c}})$. By  Lemma~\ref{lem2.4}, this occurs if and only if $I\cup J=K\cup L$ and $I\cap J=K\cap L$. Hence we are done.
\end{proof}

\subsection{Hodge relations between  periods of generalized anti-Weyl CM abelian varieties}

We start with the following lemma.
\begin{lem}\label{LemmaHodgeCycleGenerator}
The Hodge relations generated by $(2,2)$-Hodge cycles on $A$ contain the relation 
\[
\Theta_I \Theta_{\emptyset}^{r-1} = \Theta_{\{i_1\}} \cdots \Theta_{\{i_r\}}
\]
for each $I = \{i_1,\ldots,i_r\} \subseteq \{1,\ldots,g\}$ with $r \ge 2$.
\end{lem}
\begin{proof}
The proof is inspired by the proof of Corollary~\ref{CorThetaIGenerator}.

Let us do induction on $r$. When $r=2$, by Theorem~\ref{TheoremHodge22AntiWeyl} we have that $\varepsilon_I \wedge \varepsilon_{\emptyset} \wedge \varepsilon_{\{i_1\}^{\mathrm{c}}} \wedge \varepsilon_{\{i_2\}^{\mathrm{c}}}$ is a $(2,2)$-Hodge cycle (base changed to $\CC$) on $A$. Hence we obtain the desired Hodge relation $\Theta_I \Theta_{\emptyset} = \Theta_{\{i_1\}}\Theta_{\{i_2\}}$.  

Assume the result is proved for $r-1 \ge 2$. Take $I$ with $|I|=r \ge 3$. Then by applying  Theorem~\ref{TheoremHodge22AntiWeyl} to $A$ and the quadruple $I, \emptyset, \{i_1\}, I\setminus \{i_1\}$, we obtain a Hodge relation in degree $2$ 
\[
\Theta_I \Theta_{\emptyset} = \Theta_{\{i_1\}}\Theta_{I\setminus\{i_1\}}.
\]
By induction hypothesis applied to $I\setminus\{i_1\}$, whose cardinality is $r-1 \ge 2$, we get a Hodge relation generated by those given by $(2,2)$-Hodge cycles on $A$
\[
\Theta_{I \setminus \{i_1\}}  \Theta_{\emptyset}^{ r-2} = \Theta_{\{i_2\}} \cdots \Theta_{\{i_r\}}.
\]
Thus we can conclude by multiplying both Hodge relations above.
\end{proof}

Now we are ready to finish the proof of Theorem~\ref{ThmHodgeRingAntiWeyl}.
\begin{proof}[Proof of Theorem~\ref{ThmHodgeRingAntiWeyl}]
The conclusion trivially holds true if $g=1$, so we assume $g\ge 2$.

Consider a Hodge relation given by  a $(p,p)$-Hodge cycle on $A^n$ with $p \ge 3$ and $n \ge 1$. By Proposition~\ref{PropHdgRingAntiWeylPrem},  this $(p,p)$-Hodge cycle is a linear combination of cycles of the form
\[
\alpha = \varepsilon_{I_1}^{(\ell_1)} \wedge \cdots \wedge \varepsilon_{I_{2p}}^{(\ell_{2p})}
\]
for subsets $I_1^{(\ell_1)} ,\ldots,I_{2p}^{(\ell_{2p})} \subseteq \{1,\ldots,g\}$ with each $j \in \{1,\ldots,g\}$ appearing exactly $p$ times in $I_1^{(\ell_1)},\ldots,I_{2p}^{(\ell_{2p})}$. In particular, $\sum_{l=1}^{2p} |I_l^{(\ell_l)}| = pg$.

Without loss of generality, we assume that $1 \not\in I_1^{(\ell_1)},\ldots,I_p^{(\ell_p)}$ and $1 \in I_{p+1}^{(\ell_{p+1})},\ldots,I_{2p}^{(\ell_{2p})}$. Thus the Hodge relation given by $\alpha$ is
\begin{equation}\label{EqHodgeRelation1}
\Theta_{I_1}\cdots \Theta_{I_p} = \Theta_{I_{p+1}^{\mathrm{c}}}\cdots \Theta_{I_{2p}^{\mathrm{c}}}
\end{equation}
with all periods being holomorphic.

We may assume that  $I_l^{(\ell_l)} \not= (I_{l'}^{(\ell_{l'})})^{\mathrm{c}}$ for any $l, l' \in \{1,\ldots,2p\}$. This is equivalent to saying that $\alpha$ is not the wedge product of a $(1,1)$-cycle with a $(p-1,p-1)$-cycle $\alpha'$. 

Applying Lemma~\ref{LemmaHodgeCycleGenerator} to each $I_1^{(\ell_1)},\ldots,I_{2p}^{(\ell_{2p})}$ (if $|I_l^{(\ell_l)}|=1$ then we do nothing), we obtain $2p$ equalities up to $\IQbar^*$ which are either trivial equalities or  Hodge relations generated by those in degree $2$. Multiplying these $2p$ equalities, we then have by Proposition~\ref{PropHdgRingAntiWeylPrem} 
\[
\Theta_{I_1} \cdots \Theta_{I_{2p}} \cdot \Theta_{\emptyset}^{p(g-2)} = (\Theta_{\{1\}}\cdots \Theta_{\{g\}})^p,
\]
which is by construction a Hodge relation generated by those in degree $2$. Multiplying $\Theta_{I_{p+1}^{\mathrm{c}}}\cdots \Theta_{I_{2p}^{\mathrm{c}}} \Theta_{\{2,\ldots,g\}}^p$ and dividing $(2\pi i)^p$ on both sides, we get 
\begin{equation}\label{EqHodgeRelation2}
\Theta_{I_1}\cdots \Theta_{I_p} \cdot (\Theta_{\emptyset}^{g-2} \Theta_{\{2,\ldots,g\}})^p = \Theta_{I_{p+1}^{\mathrm{c}}}\cdots \Theta_{I_{2p}^{\mathrm{c}}} (\Theta_{\{2\}} \cdots \Theta_{\{g\}})^p.
\end{equation}
Here we use the Hodge relation in degree $1$: $\Theta_I \Theta_{I^{\mathrm{c}}} = 2\pi i$ for all $I \subseteq \{1,\ldots,g\}$.

Comparing \eqref{EqHodgeRelation1} and \eqref{EqHodgeRelation2}, we are done if $g = 2$, and if $g\ge 3$  it suffices to prove that
\[
\Theta_{\emptyset}^{g-2} \Theta_{\{2,\ldots,g\}} = \Theta_{\{2\}} \cdots \Theta_{\{g\}}
\]
is a Hodge relation generated by those in degree $2$. 
But this follows again from Lemma~\ref{LemmaHodgeCycleGenerator} applied to $I = \{2,\ldots,g\}$. Hence we are done.
\end{proof}

\subsection{Rational $(2,2)$-Hodge cycles}
We remark that for any $\lambda \in E^c$ and any $I,J,K,L \subseteq \{2,\ldots,g\}$ satisfying \eqref{EqConditionIJKL}, the vector 
\begin{equation}\label{HV}
\sum_{\theta\in G} \theta(\lambda) \theta(\varepsilon_I^{(\ell_1)}\wedge \varepsilon_J^{(\ell_2)}\wedge \varepsilon_{K^{\mathrm{c}}}^{(\ell_3)}\wedge \varepsilon_{L^{\mathrm{c}}}^{(\ell_4)})
\end{equation}
is a Hodge cycle in $B^2(A^n)$. Moreover this Hodge cycle is not in the space generated by the algebraic cycles of codimension 1 if $\{I,J\}\not=\{K,L\}$. It is also not known \`a priori that this Hodge cycle comes from an algebraic cycle.  To clarify the situation let us describe the Galois orbit of a vector of the form 
$\varepsilon_I^{(\ell_1)}\wedge \varepsilon_J^{(\ell_2)}\wedge \varepsilon_{K^{\mathrm{c}}}^{(\ell_3)}\wedge \varepsilon_{L^{\mathrm{c}}}^{(\ell_4)}$.

\begin{defini}
Let  $\sum_{\theta\in G} \theta(\lambda) \theta(\varepsilon_I^{(\ell_1)}\wedge \varepsilon_J^{(\ell_2)}\wedge \varepsilon_{K^{\mathrm{c}}}^{(\ell_3)}\wedge \varepsilon_{L^{\mathrm{c}}}^{(\ell_4)})$ be a Hodge cycle  as in \eqref{HV}, with $\lambda\in E^c$ and $I,J,K,L \subseteq \{2,\dots,g\}$. We define the \textbf{support} of such a Hodge cycle to be the set of quadraples $(\theta\cdot I, \theta\cdot J, \theta\cdot K^{\mathrm{c}}, \theta\cdot L^{\mathrm{c}})$ 
when $\theta$ varies in $G$. 

We say that two such Hodge cycles are \textbf{equivalent} if they have the same support. For $I,J,K,L \subseteq \{1,\dots,g\}$, we denote by $\mathrm{HC}(I,J,K,L)$ the $\QQ$-vector subspace of $B^2(A^n)$ of Hodge cycles equivalent to $\sum_{\theta\in G}  \theta(\varepsilon_I^{(\ell_1)}\wedge \varepsilon_J^{(\ell_2)}\wedge \varepsilon_{K^{\mathrm{c}}}^{(\ell_3)}\wedge \varepsilon_{L^{\mathrm{c}}}^{(\ell_4)})$. 
\end{defini}

\begin{rem}
Notice that the group $E^\times$ acts naturally on $\mathrm{HC}(I,J,K,L)$ as the action of the endomorphisms $E$ on $H^4(A^n ,\CC)$ is diagonalized in the base 
$$
\cB:=\Big\{\varepsilon_{B_1}^{(\ell_1)} \wedge \varepsilon_{B_2}^{(\ell_2)} \wedge \varepsilon_{B_3}^{(\ell_3)} \wedge \varepsilon_{B_4}^{(\ell_4)} , \mbox{ with } B_1,B_2,B_3,B_4 \mbox{ subsets of } \{1,\dots,g\}\Big\}
$$ 
Each $\alpha\in E^\times$ acts on $\varepsilon_{B_1}^{(\ell_1)} \wedge \varepsilon_{B_2}^{(\ell_2)} \wedge \varepsilon_{B_3}^{(\ell_3)} \wedge \varepsilon_{B_4}^{(\ell_4)}$ by multiplication by $\varepsilon_{B_1}(\alpha)\varepsilon_{B_2}(\alpha)\varepsilon_{B_3}(\alpha)\varepsilon_{B_4}(\alpha)$.
The remark is a consequence of the fact that for any $\sigma\in G$, any subset $B$ of $\{1,\dots,g\}$ and any $\alpha\in E^\times$, we have  $\sigma (\alpha \cdot \varepsilon_B^{(\ell)})=\sigma(\alpha) \varepsilon_{\sigma(B)}^{(\ell)}$ for any $\ell \in \{1,\ldots,n\}$.
\end{rem}

Here is a result on the $(2,2)$-Hodge cycle (defined over $\QQ$) on anti-Weyl CM abelian varieties.
\begin{prop}
Assume $A$ is anti-Weyl, \textit{i.e.} $G = (\ZZ/2\ZZ)^g\rtimes S_g$. 
For any $I,J,K,L\subseteq \{2,\ldots,g\}$ satisfying \eqref{EqConditionIJKL}, the vector $\varepsilon_{IJKL}:=\varepsilon_I\wedge \varepsilon_J\wedge \varepsilon_{K^{\mathrm{c}}}\wedge \varepsilon_{L^{\mathrm{c}}}$ is in the Galois orbit of a unique vector of the form
$$
\varepsilon_{\emptyset}\wedge \varepsilon_{\{2,\dots,r\}}\wedge \varepsilon_{\{2,\dots,s\}^{\mathrm{c}}}\wedge \varepsilon_{\{s+1,\dots,r\}^{\mathrm{c}}}
$$
for an integer $r$ such that  $2\leq r\leq g$ and an integer $s$ such that $2\leq s\leq \frac{r}{2}$.
\end{prop}
\begin{proof}
Applying some element of $(\ZZ/2\ZZ)^g < G$, we see that $\varepsilon_{IJKL}$ is Galois conjugate to a vector of the form $\varepsilon_{\emptyset}\wedge \varepsilon_{J_1}\wedge \varepsilon_{K_1^{\mathrm{c}}}\wedge \varepsilon_{L_1^{\mathrm{c}}}$ 
for some $J_1,K_1,L_1 \subseteq\{2,\ldots,g\}$ such that $J_1=K_1\cup L_1$ and $K_1\cap L_1=\emptyset$. If $\vert J_1\vert= r-1$, we can apply a $\theta\in \{0\}\times S_g$ to see that $\varepsilon_{IJKL}$ is Galois conjugate to a vector of the form $\varepsilon_{\emptyset} \wedge \varepsilon_{\{2,\dots,r\}}\wedge \varepsilon_{K_2^{\mathrm{c}}}\wedge \varepsilon_{L_2^{\mathrm{c}}}$ for some subsets $K_2, L_2$ such that $\{2,\dots,r\}=K_2\cup L_2$ and $K_2\cap L_2=\emptyset$. We can still use a element of $S_g$ with support in $\{2,\dots,r\}$ to conclude that $\varepsilon_{IJKL}$ is Galois conjugate to a vector of the required form. 

Let  $\varepsilon_{\emptyset}\wedge \varepsilon_{\{2,\dots,r\}}\wedge \varepsilon_{\{2,\dots,s\}^{\mathrm{c}}}\wedge \varepsilon_{\{s+1,\dots,r\}^{\mathrm{c}}}$ and $ \varepsilon_{\emptyset}\wedge \varepsilon_{\{2,\dots,r'\}}\wedge \varepsilon_{\{2,\dots,s'\}^{\mathrm{c}}}\wedge \varepsilon_{\{s'+1,\dots,r'\}^{\mathrm{c}}}$ be two vectors with $r,s$ and $r',s'$ satisfying the inequalities of the proposition. Assume that they are Galois conjugate by an element $\theta \in G$. Then $\theta$ has to be in $\{0\}\times S_g$ to preserve the first component $\varepsilon_\emptyset$. Next $\theta\in \{0\}\times S_g$ has to change $\{2,\dots,r\}$ into $\{2,\dots,r'\}$. This implies that $ r=r'$ and that we may assume the support of $\theta$ to be $\{2,\dots,r\}$. We then remark that, as $\theta$ in $S_g$ preserves the size of the sets in $\cP(\{1,\ldots,g\})$,  we must have $\{s,r-s\}=\{s',r-s'\}$. This finishes the proof of the proposition.
\end{proof}

\section{Example of a cyclotomic extension}\label{SectionExample}
Let $E = \QQ(\mu_{19})$. Then $E/\QQ$ is Galois with Galois group $G$ is $(\ZZ/19\ZZ)^* \cong \ZZ/18\ZZ$, and $\Hom(E,\CC) = \Hom(E, E) \cong G$.

The inclusion  $G<(\ZZ/2\ZZ)^9 \rtimes S_9$ from \cite[Imprimitivity Theorem]{DodsonThe-structure-o} is given by $G = \ZZ/2\ZZ \times \ZZ/9\ZZ$. Thus the complex conjugation on $ \Hom(E,\CC) \cong \ZZ/18\ZZ$ is $[9]$. 

Now
\[
\Phi := \{[0],[2],[3],[6],[10],[13],[14],[16],[17]\}.
\]
is a CM type on $E$. Let $A:=A_{(E,\Phi)}$ be the abelian variety associated with $(E,\Phi)$, and for each $[a] \in \ZZ/18\ZZ \cong \Hom(E,\CC)$ denote by $\Theta_{[a]}$  the period of $A$ corresponding to $[a]$. Then the holomorphic periods of $A$ are the $\Theta_{[a]}$'s with $[a] \in \Phi$.

\subsection{(Cubic) relations from the kernel}
We show that there are two non-trivial cubic relations between the holomorphic periods of $A$  which cannot be generated by algebraic relations of smaller degrees \textit{if we work only with powers of $A$}.

Define the pairing $\langle \cdot, \cdot \rangle$ on the $\ZZ$-module $\bigoplus_{[a]\in \ZZ/18\ZZ} \ZZ[a]$ by 
\[
\langle [a], [b]\rangle = \begin{cases} 1 & \text{ if } [a]=[b] \\ -1 & \text{ if } [a]=[9+b] \\ 0 & \text{ otherwise}. \end{cases}
\]

Consider the first short exact sequence in \eqref{EqSES} for our situation. For the $\ZZ$-module $\ZZ[0]\oplus \ZZ[2]\oplus \ZZ[3]\oplus \ZZ[6] \oplus \ZZ[10] \oplus \ZZ[13] \oplus \ZZ[14] \oplus \ZZ[16] \oplus \ZZ[17]$, an element $v:=a_{[0]}[0]+a_{[2]}[2]+a_{[3]}[3]+a_{[6]}[6]+a_{[10]}[10]+a_{[13]}[13]+a_{[14]}[14]+a_{[16]}[16]+a_{[17]}[17]$ is in the kernel $N$ if and only if $\langle v, [a]\cdot \sum_{[b] \in \Phi}[b] \rangle = 0$ for all $[a]\in \ZZ/18\ZZ$, \textit{i.e.} if and only if $\langle v, \sum_{[b]\in \Phi}[a+b] \rangle = 0$ for all $[a]\in \ZZ/18\ZZ$, hence if and only if
\[
\begin{cases}
a_{[0]} +a_{[2]} +a_{[3]} +a_{[6]} +a_{[10]} +a_{[13]} +a_{[14]} +a_{[16]} +a_{[17]} = 0 \\
a_{[0]} -a_{[2]} +a_{[3]} -a_{[6]} -a_{[10]} -a_{[13]} +a_{[14]} -a_{[16]} +a_{[17]} = 0 \\
a_{[0]} +a_{[2]} -a_{[3]} -a_{[6]} -a_{[10]} -a_{[13]} -a_{[14]} +a_{[16]} -a_{[17]} = 0 \\
-a_{[0]} +a_{[2]} +a_{[3]} +a_{[6]} -a_{[10]} +a_{[13]} -a_{[14]} +a_{[16]} +a_{[17]} = 0 \\
a_{[0]} +a_{[2]} +a_{[3]} +a_{[6]} +a_{[10]} -a_{[13]} +a_{[14]} -a_{[16]} +a_{[17]} = 0 \\
a_{[0]} -a_{[2]} +a_{[3]} -a_{[6]} -a_{[10]} -a_{[13]} -a_{[14]} -a_{[16]} -a_{[17]} = 0 \\
-a_{[0]} +a_{[2]} -a_{[3]} +a_{[6]} -a_{[10]} -a_{[13]} -a_{[14]} +a_{[16]} -a_{[17]} = 0 \\
-a_{[0]} +a_{[2]} +a_{[3]} +a_{[6]} +a_{[10]} +a_{[13]} -a_{[14]} -a_{[16]} +a_{[17]} = 0 \\
a_{[0]} -a_{[2]} +a_{[3]} +a_{[6]} +a_{[10]} -a_{[13]} +a_{[14]} -a_{[16]} -a_{[17]} = 0 
\end{cases}
\]
if and only if
\[
\begin{cases}
a_{[0]} = -a_{[3]} = a_{[6]} \\
a_{[2]} = a_{[14]} = -a_{[17]} \\
a_{[10]} = -a_{[13]} = a_{[16]} \\
a_{[0]} + a_{[2]} + a_{[10]} = 0.
\end{cases}
\]
Thus the kernel $N$ is generated by $[0]-[2]-[3]+[6]-[14]+[17]$ and $[0]-[3]+[6]-[10]+[13]-[16]$. Hence we find cubic relations of holomorphic periods of $A$:
\begin{equation}\label{EqHoloAlgRelEg}
\Theta_{[0]}\Theta_{[6]}\Theta_{[17]} = \Theta_{[2]}\Theta_{[3]}\Theta_{[14]}\quad \text{ and }\quad \Theta_{[0]}\Theta_{[6]}\Theta_{[13]} = \Theta_{[3]}\Theta_{[10]}\Theta_{[16]}.
\end{equation}

\subsection{Finding compagnons of $A$ to get quadratic relations}
Now we show that the cubic relations in \eqref{EqHoloAlgRelEg} are generated by quadratic relations between holomorphic periods of $A$ and its compagnons.

Set $\Phi_E$ to be the reflex CM type of $\Phi$, \textit{i.e.}
\begin{align*}
\Phi_E := \Phi^* & = \{[0], [-2], [-3], [-6], [-10], [-13], [-14], [-16], [-17]\} \\
& = \{[0], [1], [2], [4], [5], [8], [12], [15], [16]\}.
\end{align*}
Call $\phi_1 = [0]$, $\phi_2 = [1]$, $\phi_3 = [2]$, $\phi_4 = [4]$, $\phi_5 = [5]$, $\phi_6 = [8]$, $\phi_7 = [12]$, $\phi_8 = [15]$, $\phi_9 = [16]$. In the terminology of generalized anti-Weyl CM abelian varieties (more precisely the identification \eqref{EqBijectionsSubsetsZ2ZCMtypes}), $\Phi_E$ corresponds to $\emptyset \in \cP( \{1,\ldots,9\})$ and $\Phi$ corresponds to $\{2, 4, 5, 6, 7, 8\} \in \cP( \{1,\ldots,9\})$. 

Introduce the following notation. For each  $[a]\in G = \ZZ/18\ZZ$, let $I([a]) $ be the subset of $\{1,\ldots,9\}$ corresponding to the CM type $[a]\cdot\Phi_E$ on $E$ under the identification \eqref{EqBijectionsSubsetsZ2ZCMtypes}.

Let us compute the $G$-orbit of $\emptyset \in \cP(\{1,\ldots,9\})$. To do this, we compute the $G$-orbit of $\Phi_E \in \{\text{CM types on }E\}$ as follows:
\begin{align*}
[0] \cdot \Phi_E = \{[0], [1], [2], [4], [5], [8], [12], [15], [16]\} & = \{\phi_1,\phi_2,\phi_3,\phi_4,\phi_5,\phi_6,\phi_7,\phi_8,\phi_9\} \\ 
[1] \cdot \Phi_E =  \{[1], [2], [3], [5], [6], [9], [13], [16], [17]\} & = \{\phi_2, \phi_3, \phi_5, \phi_9 \} \bigcup \{\overline{\phi_1}, \overline{\phi_4}, \overline{\phi_6}, \overline{\phi_7}, \overline{\phi_8} \}  \\
[2] \cdot \Phi_E = \{[2], [3], [4], [6], [7], [10], [14], [17], [0]\} &  =  \{\phi_1, \phi_3, \phi_4 \} \bigcup \{\overline{\phi_2}, \overline{\phi_5}, \overline{\phi_6}, \overline{\phi_7}, \overline{\phi_8},  \overline{\phi_9} \}   \\
\vdots & \\
[9+a]\cdot \Phi_E = \overline{[a] \cdot \Phi_E}
\end{align*}
Thus the $G$-orbit of $\emptyset$ is:
\[
\begin{array}{llll}
 I([0]) = \emptyset, & I([1]) = \{1, 4, 6, 7, 8\}, &  I([2]) =\{2, 5, 6, 7, 8, 9\}, \\
  I([3]) = \{3, 7, 9\}, &
  I([4]) =\{1, 8\}, &   I([5]) =\{1, 2, 4, 6, 7, 8, 9\}, \\
      I([6]) =\{2, 3, 5, 7, 8, 9\}, & I([7]) = \{1, 3, 9\}, &
    I([8]) = \{1, 2, 4, 8\}, \\
    I([9]) =  \{1,2,\ldots,9\}, &  I([10]) =  \{2, 3, 5, 9\},  &   I([11]) =  \{1, 3, 4\},  \\
      I([12]) = \{1, 2, 4, 5, 6, 8\}, &    I([13]) = \{2, 3, 4, 5, 6, 7, 9\}, & I([14]) = \{3, 5\}, \\
   I([15]) = \{1, 4, 6\}, &   I([16]) = \{2, 4, 5, 6, 7, 8\}, &
      I([17]) = \{3, 5, 6, 7, 9\}.
 \end{array}
\]
Notice that the CM type $\Phi$, which is the reflex of $\Phi_E $,  can be recovered by  $\Phi = \{[a] \in \ZZ/18\ZZ : 1 \not\in I([a]) \}$. 
The period $\Theta_{[a]}$ of $A=A_{(E,\Phi)}$ is $\Theta_{I([a])}$ in the terminology of generalized anti-Weyl CM abelian varieties.

Let us look at the first cubic relation in \eqref{EqHoloAlgRelEg}. Let $L :=\{5,6\} $. Then $I([0]) \cup I([17]) = I([3]) \cup L$ and $I([0]) \cap I([17]) = I([3]) \cap L$, and $I([6]) \cup L = I([2]) \cup I([14])$ and $I([6]) \cap L = I([2]) \cap I([14])$. Thus the first cubic relation in \eqref{EqHoloAlgRelEg} is generated by
\[
\Theta_{[0]}\Theta_{[17]} = \Theta_{[3]}\Theta_L \quad\text{ and } \quad \Theta_{[2]}\Theta_{[14]} = \Theta_{[6]}\Theta_L,
\]
with $\Theta_L$ a holomorphic period of a compagnon $A_L$ of $A$. 

We can compute $A_L$ as follows. Before the computation, notice that $G \cdot L \not= G\cdot \emptyset$ because $L \not\in G\cdot \emptyset$. So $A_L$ is not isogeneous to $A$.

The CM type corresponding to $L$ under the identification \eqref{EqBijectionsSubsetsZ2ZCMtypes} is
\[
\{\phi_1, \phi_2, \phi_3, \phi_4, \overline{\phi_5}, \overline{\phi_6}, \phi_7, \phi_8, \phi_9\} = \{[0], [1], [2], [4], [14], [17], [12], [15], [16]\}.
\]
Thus the CM type associated with $A_L$ is
\begin{align*}
\Phi_L:=\left\{ [a] \in \ZZ/18\ZZ : [0] \not\in [a] \cdot\{[0], [1], [2], [4], [14], [17], [12], [15], [16]\} \right\} \\
 = \{[5], [7], [8], [9], [10], [11], [12], [13], [15]\},
\end{align*}
\textit{i.e.} $A$ is the CM abelian variety associated with the CM pair $(\QQ(\mu_{19}), \Phi_L)$.

Similarly, the second cubic relation in \eqref{EqHoloAlgRelEg} is generated by 
\[
\Theta_{[0]}\Theta_{[13]} = \Theta_{[10]}\Theta_{L'} \quad\text{ and }\quad \Theta_{[3]}\Theta_{[16]} = \Theta_{[6]}\Theta_{L'}
\]
with $L' = \{4,6,7\}$, and the compagnon $A_{L'}$ (corresponding to $G\cdot L'$) of $A$ is associated with the CM pair $(\QQ(\mu_{19}), \Phi_{L'})$ with $\Phi_{L'} = \{[4], [6], [7], [8], [9], [10], [11], [12], [14]\}$.

\section{Shimura subvariety giving quadratic relations}\label{SectionSubShimuraVariety}
This section aims to explain a link between the quadratic relations between the holomorphic periods explained in this paper and the theory of bi-$\IQbar$ decomposition of Shimura varieties which we developed in \cite{GUY}. We will prove Conjecture~\ref{ConjBiQbar} for generalized anti-Weyl CM abelian varieties.

Here is our set up. Let $(E,\Phi_E)$ be a CM type with $E$ a CM field of degree $2g$. Let $A$ be a generalized anti-Weyl CM abelian variety arising from $(E,\Phi_E)$. Then $\dim A = 2^{g-1}$. 

Let $\mathbb{A}_{2^{g-1}}$ be the moduli space of principally polarized abelian varieties of dimension $2^{g-1}$, and let $[o] \in \mathbb{A}_{2^{g-1}}(\IQbar)$ parametrize $A$.

The goal is to construct a Shimura subvariety of $\mathbb{A}_{2^{g-1} r}$ for some $r \le g!$, passing through $[\mathbf{o}]:=([o],\ldots,[o]) \in \mathbb{A}_{2^{g-1} r}(\IQbar)$, from which we can read off  non-trivial elementary quadratic relations among the holomorphic periods of $A_o$. More precisely, we prove:

\begin{teo}\label{ThmConjAntiWeylGen}
Assume $g \ge 3$. There exist an integer $r \in \{1,\ldots,g!\}$ and a Shimura subvariety $S$ of $\mathbb{A}_{2^{g-1}r}$, passing through $[\mathbf{o}]:=([o],\ldots,[o]) \in \mathbb{A}_{2^{g-1} r}(\IQbar)$, with the following property: $T_{[\mathbf{o}]}S$ is not the direct sum of root spaces of $T_{[o]}\mathbb{A}_{2^{g-1}}$.
\end{teo}
This proves Conjecture~\ref{ConjBiQbar} for generalized anti-Weyl CM abelian varieties if $g\ge 3$, together with Theorem~\ref{TheoremHodge22AntiWeyl} and \eqref{EqConditionIJKL}.

\subsection{Setup for root spaces}\label{SubsectionSetupEigenspace}
Let $o \in \mathfrak{H}_{2^{g-1}}$ be a point whose image is $[o]$ under the uniformizing map $\mathfrak{H}_{2^{g-1}} \rightarrow \mathbb{A}_{2^{g-1}}$. 

Let $T_o = \mathrm{MT}(o)$ and let $T$ be a maximal torus of $\mathrm{GSp}_{2^g}$ which contains $T_o$.

Fix an order on $\cP(\{1,\ldots,g\})$ such that $I \le J \Rightarrow I^{\mathrm{c}} \ge J^{\mathrm{c}}$ and $I < J$ for all $1\not\in I$ and $1 \in J$. The roots of $(T,\mathrm{GSp}_{2^g})$ are $\pm e_I \pm e_J$ for all subsets $I \le J$ of $ \{2,\ldots,g\}$.

Since $e_I + e_{I^{\mathrm{c}}}$ is constant for all $I \subseteq \{2,\ldots,g\}$, we can write the roots of $(T,\mathrm{GSp}_{2^g})$ as $e_I + e_J$ for all subsets $I \le J$ of $ \{1,\ldots,g\}$. 

For $I\le J$, denote by $V_{I,J}$  the root space for $e_I+e_J$; it has dimension $1$. By abuse of notation, if $I \ge J$, set $V_{I,J} := V_{J,I}$. 

The complex conjugation sends $V_{I,J}$ to $V_{I^{\mathrm{c}},J^{\mathrm{c}}}$ for each pair $(I,J)$. For $I, J \subseteq \{2,\ldots,g\}$, take $E_{I,J}$ to be a non-zero vector in $V_{I,J}$; then $E_{I,J}$ is an eigenvector for $e_I+e_J$. For $I, J \subseteq \{1,\ldots,g\}$ which contain $1$, take $E_{I,J}$ to be the complex conjugation of $E_{I^{\mathrm{c}},J^{\mathrm{c}}}$; then $E_{I,J}$ is an eigenvector for $e_I+e_J$.

General theory of reductive groups says that $V_{I,J}$ and $V_{I^{\mathrm{c}},J^{\mathrm{c}}}$ generate a  $\mathfrak{sl}_2$-triple. Hence we can normalize the choices of $E_{I,J}$ in the previous paragraph such that $E_{I,J}$ and $E_{I^{\mathrm{c}},J^{\mathrm{c}}}$ give a $\mathfrak{sl}_2$-triple for all $I, J \subseteq \{2,\ldots,g\}$, \textit{i.e.}
\begin{equation}\label{EqChoiceOfSL2Triple}
[E_{I,J},[E_{I,J},E_{I^{\mathrm{c}},J^{\mathrm{c}}}]] = 2E_{I,J} \quad \text{ and }\quad [E_{I^{\mathrm{c}},J^{\mathrm{c}}}, [E_{I^{\mathrm{c}}, J^{\mathrm{c}}}, E_{I,J}]] = 2E_{I^{\mathrm{c}},J^{\mathrm{c}}}.
\end{equation}

Finally, for each subset $\bU  \subseteq \{2,\ldots,g\}$, define
\[
V_{\bU} := \sum_{\emptyset \subseteq I \subseteq \bU} V_{I, \{2,\ldots,g\}\setminus I} = \begin{cases} \bigoplus_{\emptyset \subseteq I \subseteq \bU} V_{I,\{2,\ldots,g\}\setminus I} & \text{if }\bU\not=\{2,\ldots,g\} \\
 \bigoplus_{\substack{\emptyset \subseteq I \subseteq \{2,\ldots,g\} \\ I \le \{2,\ldots,g\}\setminus I}} V_{I,\{2,\ldots,g\}\setminus I} & \text{if }\bU=\{2,\ldots,g\}  \end{cases}.
\]

\subsection{Construction of $\mathfrak{sl}_2$-triples}

For each subset $\bU  \subseteq \{2,\ldots,g\}$, set
\begin{equation}
v_{\bU}:= \epsilon_{\bU} \sum_{\emptyset \subseteq I \subseteq \bU} E_{I,\{2,\ldots,g\}\setminus I},
\end{equation}
where $\epsilon_{\bU} = 1/2$ if $\bU = \{2,\ldots,g\}$ and $\epsilon_{\bU} = 1$ otherwise. In other words, we have
\begin{equation}
v_{\bU} = \begin{cases} \sum_{\emptyset \subseteq I \subseteq \bU} E_{I,\{2,\ldots,g\}\setminus I} &  \text{if }\bU\not=\{2,\ldots,g\} \\ \sum_{\substack{\emptyset \subseteq I \subseteq \{2,\ldots,g\} \\ I \le \{2,\ldots,g\}\setminus I}} E_{I,\{2,\ldots,g\}\setminus I} &  \text{if }\bU=\{2,\ldots,g\}  \end{cases}.
\end{equation}
Next, set
\begin{equation}
\bar{v}_{\bU} := \epsilon_{\bU}\sum_{I' \supseteq \bU^{\mathrm{c}} } E_{I', \{1\}\cup (I')^{\mathrm{c}}} = \begin{cases} \sum_{I' \supseteq \bU^{\mathrm{c}} } E_{I', \{1\}\cup (I')^{\mathrm{c}}}  &  \text{if }\bU\not=\{2,\ldots,g\} \\ \sum_{\substack{1\in I' \\ I'  \ge  \{1\}\cup (I')^{\mathrm{c}}}} E_{I', \{1\}\cup (I')^{\mathrm{c}}} &  \text{if }\bU=\{2,\ldots,g\}  \end{cases}. 
\end{equation}

The goal of this subsection is to prove the following proposition.
\begin{prop}\label{PropSL2Triple}
The following are true:
\begin{enumerate}
\item[(i)] The complex conjugation of $v_{\bU}$ is $\bar{v}_{\bU}$.
\item[(ii)] The triple $(v_{\bU}, \bar{v}_{\bU}, [v_{\bU}, \bar{v}_{\bU}])$ is an $\mathfrak{sl}_2$-triple.
\end{enumerate}
\end{prop}

We start with the following preparation.
\subsubsection{Pairing of $v_{\bU}$ and $\bar{v}_{\bU}$}
Let $I, J, I', J' \subseteq \{1,\ldots,g\}$ such that $1 \not\in I \cup J$ and $1 \in I' \cap J'$. 
Observe that $[E_{I,J}, E_{I', J'}]$ is an eigenvector for $e_I + e_J + e_{I'} + e_{J'}$. For this Lie bracket to be non-zero, $e_I + e_J + e_{I'} + e_{J'}$ must be a root. Hence $J' = J^{\mathrm{c}}$ or $J' = I^{\mathrm{c}}$.

Now we apply this discussion to $v_{\bU}$ and $\bar{v}_{\bU}$. Assume $E_{I,J}$ is a term in the sum defining $v_{\bU}$ and $E_{I',J'}$ is a term in the sum defining $\bar{v}_{\bU}$. Then  $J = \{2,\ldots,g\} \setminus I = \{2,\ldots,g\}\cap I^{\mathrm{c}}$ and $J' = \{1\} \cup (I')^{\mathrm{c}}$. 

If $J' = J^{\mathrm{c}}$, then 
\[
I' = \{1\} \cup J = \{1\} \cup ( \{2,\ldots,g\} \cap I^{\mathrm{c}}) = I^{\mathrm{c}},
\]
and hence the corresponding root is $0$.

If $J' = I^{\mathrm{c}}$, then similarly $I' = J^{\mathrm{c}}$ and the corresponding root is $0$.

Thus we have
\begin{equation}\label{EqvUvU}
[v_{\bU},\bar{v}_{\bU}] = \begin{cases} \sum_{\emptyset\subseteq I \subseteq \bU} [E_{I, \{2,\ldots,g\} \setminus I}, E_{I^{\mathrm{c}}, I\cup \{1\}}] &  \text{if }\bU\not=\{2,\ldots,g\} \\ \sum_{\substack{\emptyset\subseteq I \subseteq \bU \\ I \le \{2,\ldots,g\}\setminus I}} [E_{I, \{2,\ldots,g\} \setminus I}, E_{I^{\mathrm{c}}, I\cup \{1\}}]   & \text{if }\bU=\{2,\ldots,g\} \end{cases},
\end{equation}
which is in the Lie algebra of the maximal torus (since the corresponding root is $0$).

\begin{lem}\label{LemIIprime}
Let $I, K \subseteq \{2,\ldots,g\}$. Assume $[E_{K,\{2,\ldots,g\}\setminus K}, [E_{I, \{2,\ldots,g\} \setminus I}, E_{I^{\mathrm{c}}, I\cup \{1\}}]] \not= 0$. 
Then $K = I$ or $\{2,\ldots,g\}\setminus K = I$, and 
\begin{equation}\label{EqLiebra3}
[E_{K,\{2,\ldots,g\}\setminus K}, [E_{I, \{2,\ldots,g\} \setminus I}, E_{I^{\mathrm{c}}, I\cup \{1\}}]] = 2 E_{I,\{2,\ldots,g\}\setminus I}.
\end{equation}
\end{lem}
\begin{proof}
First, observe that $[E_{K,\{2,\ldots,g\}\setminus K}, E_{I, \{2,\ldots,g\} \setminus I}] = 0$ because both vectors are in $\mathfrak{m}^+$ and $\mathfrak{m}^+$ is abelian. Hence by the Jacobi identity, we have
\[
[E_{K,\{2,\ldots,g\}\setminus K}, [E_{I, \{2,\ldots,g\} \setminus I}, E_{I^{\mathrm{c}}, I\cup \{1\}}]] = [E_{I, \{2,\ldots,g\} \setminus I}, [E_{K, \{2,\ldots,g\}\setminus K}, E_{I^{\mathrm{c}}, I\cup \{1\}}]].
\]
The left hand side is not $0$ by assumption, and hence $[E_{K, \{2,\ldots,g\}\setminus K}, E_{I^{\mathrm{c}}, I\cup \{1\}}] \not= 0$. Hence $K = I$ or $\{2,\ldots,g\}\setminus K = I$ by the computation above this lemma.

Now \eqref{EqLiebra3} holds true because we have chosen $E_{I,J}$ and $E_{I^{\mathrm{c}},J^{\mathrm{c}}}$ to generate a $\mathfrak{sl}_2$-triple; see \eqref{EqChoiceOfSL2Triple}.
\end{proof}

\begin{cor}\label{CorSL2TripleFirstPartComputation}
We have
\[
\left[ v_{\bU}, [v_{\bU},\bar{v}_{\bU}] \right]  = 2 v_{\bU}.
\]
\end{cor}
\begin{proof}
Assume $\bU \not= \{2,\ldots,g\}$. 
By \eqref{EqvUvU}, we have
\begin{align*}
\left[ v_{\bU}, [v_{\bU},\bar{v}_{\bU}] \right]  & = \left[ \sum_{\emptyset \subseteq K \subseteq \bU} E_{K,\{2,\ldots,g\}\setminus K}, \sum_{\emptyset\subseteq I \subseteq \bU} [E_{I, \{2,\ldots,g\} \setminus I}, E_{I^{\mathrm{c}}, I\cup \{1\}}] \right]. 
\end{align*}
By Lemma~\ref{LemIIprime} and our assumption $\bU \not= \{2,\ldots,g\}$, the right hand equals
\[
\sum_{\emptyset \subseteq I \subseteq \bU} 2 E_{I, \{2,\ldots,g\}\setminus I} = 2v_{\bU}.
\]

Assume $\bU = \{2,\ldots,g\}$. By \eqref{EqvUvU}, we have
\begin{align*}
\left[ v_{\bU}, [v_{\bU},\bar{v}_{\bU}] \right]  & = \left[ \sum_{\substack{\emptyset\subseteq K \subseteq \bU \\ K \le \{2,\ldots,g\}\setminus K}} E_{K,\{2,\ldots,g\}\setminus K}, \sum_{\substack{\emptyset\subseteq I \subseteq \bU \\ I \le \{2,\ldots,g\}\setminus I}} [E_{I, \{2,\ldots,g\} \setminus I}, E_{I^{\mathrm{c}}, I\cup \{1\}}] \right]. 
\end{align*}
By Lemma~\ref{LemIIprime}, the right hand equals
\[
\sum_{\substack{\emptyset\subseteq I \subseteq \bU \\ I \le \{2,\ldots,g\}\setminus I}} 2 E_{I, \{2,\ldots,g\}\setminus I} = 2v_{\bU}.
\]
Now we are done.
\end{proof}

\subsubsection{Proof of Proposition~\ref{PropSL2Triple}} 
Part (i) follows immediately from the following observation. The complex conjugation sends $E_{I, J} \mapsto E_{I^{\mathrm{c}},J^{\mathrm{c}}}$ for each pair $(I,J)$. In particular for $I \subseteq \{2,\ldots,g\}$,  the pair $(I, \{2,\ldots,g\}\setminus I)$ is mapped to 
\[
(I^{\mathrm{c}}, \{1\}\cup I).
\]
Notice that this new pair is uniquely determined by the relations $I^{\mathrm{c}} \cap ( \{1\}\cup I) = \{1\}$ and $I^{\mathrm{c}} \cup ( \{1\}\cup I) = \{1,\ldots,g\}$. 

Now let us prove part (ii). We need to check:
\begin{enumerate}
\item[(a)] $[v_{\bU},v_{\bU}] = 0$ and $[\bar{v}_{\bU}, \bar{v}_{\bU}] = 0$;
\item[(b)] $[v_{\bU},\bar{v}_{\bU}]$ is in the Lie algebra of the maximal torus;
\item[(c)] $[v_{\bU}, [v_{\bU},\bar{v}_{\bU}]] = 2v_{\bU}$ and $[\bar{v}_{\bU}, [v_{\bU},\bar{v}_{\bU}]] = 2\bar{v}_{\bU}$.
\end{enumerate}

For (a), we have
\[
[v_{\bU},v_{\bU}]  = \epsilon_{\bU}^2 \left[\sum_{\emptyset\subseteq I \subseteq \bU}E_{I,\{2,\ldots,g\}\setminus I}, \sum_{\emptyset\subseteq K \subseteq \bU}E_{K,\{2,\ldots,g\}\setminus K}\right].
\]
Now that $[E_{I,\{2,\ldots,g\}\setminus I}, E_{K,\{2,\ldots,g\}\setminus K}]$ is an eigenvector for $e_I+e_{\{2,\ldots,g\}\setminus I} + e_K + e_{\{2,\ldots,g\}\setminus K}$. But $e_I+e_{\{2,\ldots,g\}\setminus I} + e_K + e_{\{2,\ldots,g\}\setminus K}$ cannot be a root since $I, K \subseteq \{2,\ldots,g\}$. So each term in the sum is $0$. Hence $[v_{\bU},v_{\bU}] = 0$. Similarly $[\bar{v}_{\bU}, \bar{v}_{\bU}] = 0$.

We proved assertion (b) below \eqref{EqvUvU}. The first equality in assertion (c) holds true by Corollary~\ref{CorSL2TripleFirstPartComputation}, and a similar computation yields the second equality in assertion (c). We are done. \qed

\subsection{Upshot on existence of subgroups and of Shimura subvarieties}
Retain the notation from $\S$\ref{SubsectionSetupEigenspace}. Take $\bU \subseteq \{2,\ldots,g\}$ such that $|\bU| \ge 2$; such a $\bU$ exists if $g \ge 3$.

Let $H$ be the subgroup of $\mathrm{GSp}_{2^g,\IQbar}$ generated by $T$, and $\exp(v_{\bU})$ and $\exp(\bar{v}_{\bU})$. By Proposition~\ref{PropSL2Triple}, $H$ is a reductive group of semi-simple rank $1$ defined over $\RR\cap \IQbar$. In fact, let $F$ be the maximal totally real subfield of $E$ and let $F^c$ be its Galois closure. Then $H$ is defined over $F^c$, and $r:=[F^c : \QQ] \le g!$. 

We thus obtain a $\QQ$-subgroup $\mathrm{Res}_{F^c/\QQ}H$ of the group $\mathbb{G}_{\mathrm{m}}\cdot \mathrm{Res}_{F^c/\QQ}\mathrm{Sp}_{2^g}$, and hence of the group $\mathrm{GSp}_{2^g r}$. 

The point $[\mathbf{o}] := ([o],\ldots,[o]) \in \mathbb{A}_{2^{g-1}r}$ has Mumford--Tate group $T_o$. Thus $H$ defines a Shimura subvariety $S$ of $\mathbb{A}_{2^{g-1}r}$ which passes through $[\mathbf{o}]$, and $\dim S = r$.

\begin{proof}[Proof of Theorem~\ref{ThmConjAntiWeylGen}]
The Cartan involution given by $o \in \mathfrak{H}_{2^{g-1}}$ induces the Cartan decomposition $\Lie \mathrm{GSp}_{2^g,\RR} = \mathfrak{k} \oplus \mathfrak{m}$ where $\mathfrak{m}$ is the eigenspace of $-1$ and $\mathfrak{k} $ is the eigenspace of $1$. This induces a canonical isomorphism of $\RR$-vector spaces $T_o\mathfrak{H}_{2^{g-1}} = \mathfrak{m}$. Furthermore, the complex structure on $\mathfrak{H}_{2^{g-1}}$ is given by an endomorphism $J$ of $T_o\mathfrak{H}_{2^{g-1}}$ such that $J^2 = - \mathrm{Id}$. Thus $J$ acts on $\mathfrak{m}_{\CC}$ and we have a decomposition $\mathfrak{m}_{\CC} = \mathfrak{m}^+ \oplus \mathfrak{m}^-$, where $J$ acts on $\mathfrak{m}^+$ by multiplication by $\sqrt{-1}$ and on $\mathfrak{m}^-$ by multiplication by $-\sqrt{-1}$. The Borel Embedding Theorem induces a canonical isomorphism $T_o \mathfrak{H}_{2^{g-1}} = \mathfrak{m}^+$ as $\CC$-vector spaces. Moreover, in our case, we have 
\[
\mathfrak{m}^+ = \bigoplus_{ \substack{1 \not\in I, J \\ I \le J}} V_{I,J} \qquad \text{and} \qquad \mathfrak{m}^- = \bigoplus_{ \substack{1 \in I, J \\ I \le J}} V_{I,J}.
\]

Now we have $T_{[o]}\mathbb{A}_{2^{g-1}} = \mathfrak{m}^+$, and hence
\[
T_{[\mathbf{o}]}\mathbb{A}_{2^{g-1}r} = \mathfrak{m}^+\oplus \cdots \oplus \mathfrak{m}^+  \text{ ($r$ copies)}.
\]
By definition of $v_{\bU}$, we have that $v_{\bU} \in \mathfrak{m}^+$ and that $V_{\bU}$ is the smallest direct sum of root spaces of $\mathfrak{m}^+$ which contains $v_{\bU}$. Notice that $\dim V_{\bU} \ge 2$ since $g \ge 3$. So $\CC v_{\bU}$ is not a root space of $\mathfrak{m}^+$.

 By construction of $S$, we have 
\[
T_{[\mathbf{o}]} S = \CC v_{\bU} \oplus \cdots \oplus \CC v_{\bU},
\]
compatible with the decomposition of $T_{[\mathbf{o}]}\mathbb{A}_{2^{g-1}r}$ above. Therefore $T_{[\mathbf{o}]}S$ is not a direct sum of roots spaces of $\mathfrak{m}^+$. We are done.
\end{proof}

\bibliographystyle{alpha}

\end{document}